\pdfoutput=1
\documentclass[a4paper]{wsbart}

\usepackage[utf8]{inputenc}
\usepackage[T1]{fontenc}
\usepackage{lmodern}

\usepackage[maxnames=99,style=trad-plain,backend=biber,block=space]{biblatex}
\DeclareRedundantLanguages{English}{english}

\usepackage{amsmath,amssymb}
\usepackage{amsthm}
\usepackage{bm,dsfont,stmaryrd}
\SetSymbolFont{stmry}{bold}{U}{stmry}{m}{n}
\renewcommand{\mathbb}{\mathds}
\usepackage{mathtools}
\usepackage{accents}
\usepackage{IEEEtrantools}

\usepackage{enumitem}

\usepackage[ruled,vlined]{algorithm2e}
\usepackage{caption}
\usepackage{booktabs}

\newcommand{\R}{\mathbb R}

\newcommand{\N}{\mathbb N}

\newcommand{\cco}{\llbracket}
\newcommand{\ccf}{\rrbracket}
\newcommand{\dd}{\mathop{}\!\mathrm{d}}
\newcommand{\1}{\mathbb{1}}

\newcommand{\TV}{\textnormal{TV}}
\newcommand{\na}{\nabla}

\DeclareMathOperator{\Law}{Law}

\DeclareMathOperator{\Expect}{\mathbb E}
\DeclareMathOperator{\Proba}{\mathbb P}

\numberwithin{equation}{section}
\newtheorem{thm}{Theorem}[section]

\newtheorem{cor}[thm]{Corollary}
\newtheorem{prop}[thm]{Proposition}
\theoremstyle{definition}

\newtheorem{assu}[thm]{Assumption}
\theoremstyle{remark}
\newtheorem{rem}[thm]{Remark}

\title{Time-uniform log-Sobolev inequalities\\and applications
to propagation of chaos}

\author[1]{Pierre Monmarché}
\author[2]{Zhenjie Ren}
\author[3]{Songbo Wang}
\affil[1]{LJLL \& LCT, Sorbonne Université, Paris, France}
\affil[2]{LaMME, Université d'Évry, Université Paris-Saclay, Évry, France}
\affil[3]{CMAP, École polytechnique, IP Paris, Palaiseau, France}

\begin{document}

\maketitle

\begin{abstract}
Time-uniform log-Sobolev inequalities (LSI) satisfied by solutions of semi-linear mean-field equations have recently appeared to be a key tool to obtain time-uniform propagation of chaos estimates. This work  addresses the more general settings of time-inhomogeneous Fokker--Planck equations.
Time-uniform LSI are obtained in two cases,
either with the bounded-Lipschitz perturbation argument
with respect to a reference measure,
or with a coupling approach at high temperature.
These arguments are then applied to mean-field equations,
where, on the one hand, sharp marginal propagation of chaos estimates are obtained in smooth cases
and, on the other hand, time-uniform global propagation of chaos is shown in the case of
vortex interactions with quadratic confinement potential on the whole space. In this second case, an important point is to establish global gradient and Hessian estimates, which is of independent interest. We prove these bounds in the more general situation of non-attractive logarithmic and Riesz singular interactions.
\end{abstract}

\section{Introduction}

We are interested in  families $(m_t)_{t\geqslant 0}$ of probability distributions solving time-inhomo\-geneous Fokker--Planck equations on $\R^d$ of the form
\begin{equation}\label{eq:FP_inhomogeneous}
      \partial_t m_t = \nabla \cdot(
  \sigma^2 \na m_t - b_{t} m_t )\,,
\end{equation}
where $\sigma^2 >0$ and $b_t :\R^d \to \R^d$ for $t\geqslant 0$. This describes the evolution of the law of the diffusion process
\begin{equation}
    \label{eq:EDS_inhomogeneous}
\dd X_t = b_{t}(X_t) \dd t + \sqrt 2\sigma \dd B_t\,,
\end{equation}
where $B$ is a standard $d$-dimensional Brownian motion. We have particularly in mind McKean--Vlasov equations, where $b_t$ is in fact a function of $m_t$ itself, namely
\begin{equation}
    \label{eq:McK-drift}
    b_t(x) = F(x,m_t)\,,
\end{equation}
for some suitable function $F$. Other examples are time-integrated McKean--Vlasov equations where $b_t(x) = F\bigl(x,\int_0^t m_s k_t(\dd s)\bigr)$ for some kernel $k_t$ (as in \cite{CdRDMTT}).

Denoting by $\mathcal C^1_\textnormal{c}(\R^d)$ the set of compactly supported $\mathcal C^1$ functions from $\R^d$ to $\R$, a probability measure $\mu$ on $\R^d$ is said to satisfy a log-Sobolev inequality (LSI) with constant $C>0$ if
\begin{equation}
\label{eq:LSI}
\forall h\in\mathcal C^1_\textnormal{c}(\R^d)~\text{with}\,
\int_{\R^d} h^2 \dd \mu=1\,, \qquad
\int_{\R^d} h^2 \ln(h^2) \dd \mu \leqslant C \int_{\R^d}|\na h|^2 \dd \mu\,.
\end{equation}
Equivalently, for all probability measure $\nu \in \mathcal P(\mathbb R^d)$
such that $\nu$ is absolutely continuous with respect to $\mu$
and $\sqrt{\dd \nu /\!\dd \mu} \in \mathcal C^1_\textnormal{c}$, we have
\[
\mathcal H (\nu | \mu) \leqslant \frac{C}{4} \mathcal I(\nu | \mu)\,,
\]
where $\mathcal H$, $\mathcal I$ are the relative entropy and Fisher information
defined respectively as follows:
\begin{align*}
\mathcal H (\nu | \mu)
&\coloneqq \int_{\R^d} \ln \frac{\dd \nu}{\dd\mu} \dd\nu\,, \\
\mathcal I (\nu | \mu)
&\coloneqq \int_{\R^d} \biggl| \nabla \ln \frac{\dd\nu}{\dd\mu}\bigg|^2\dd\nu\,.
\end{align*}
We want to determine suitable conditions under which the family $(m_t)_{t\geqslant0}$ solving \eqref{eq:FP_inhomogeneous} satisfies a uniform LSI, in the sense that \eqref{eq:LSI} holds with $\mu=m_t$ and a constant $C$ independent from $t$. As will be discussed below in details (in Sections~\ref{sec:MkV-smooth} and \ref{sec:MkV-log-Riesz}), for McKean--Vlasov equations, this is  an important tool to get uniform-in-time   Propagation of Chaos (PoC) estimates \cite{GLBMVortex,LLFSharp}.

The paper is organized as follows. In the rest of this introduction we state our main results concerning time-uniform LSI (Theorem~\ref{thm:high_temperature} and \ref{thm:perturbation}), which are proven in Section~\ref{sec:proofLSI}. In Section~\ref{sec:MkV-smooth} we use them to extend the range of the work  \cite{LLFSharp} of Lacker and Le Flem, obtaining sharp uniform in time PoC for McKean--Vlasov equations in cases of smooth interaction.
Section~\ref{sec:MkV-log-Riesz} addresses the question of uniform-in-time LSI and PoC for singular (log or Riesz) interactions in $\R^d$.

\bigskip

Before stating our main results, we recall first the following result of  Malrieu \cite{MalrieuLSI}, based on the classical Bakry--Émery approach.
\begin{prop}\label{prop:BakryEmery}
Assume that there exist $T > 0$, $L \in \R$ such that
for all $t\in[0,T]$ and $x$, $y\in\R^d$,
\begin{equation}
    \label{eq:curvature}
\bigl(b_t(x)-b_t(y)\bigr) \cdot (x-y) \leqslant L |x-y|^2\,,
\end{equation}
and that $m_0$ satisfies an LSI with constant $C_0>0$.
Then, for all $t\in[0,T]$, the measure $m_t$ satisfies an LSI with constant
\[C_t = e^{2Lt} C_0 + \sigma^2 \int_0^t e^{2Ls}\dd s\,.\]
\end{prop}
For completeness, the proof is recalled in Section~\ref{sec:proofLSI_BakryEmery}.

\begin{rem}\label{rem:1}
    When the curvature lower-bound $L$ in \eqref{eq:curvature} is negative, this already gives an LSI uniform in $t$, but we are mostly interested in cases where \eqref{eq:curvature} only holds with $L>0$. Nevertheless, this first proposition means that, in the next results (Theorems~\ref{thm:high_temperature} and \ref{thm:perturbation}), in fact, if the assumptions are only satisfied for $t\geqslant t_0$ for some $t_0>0$ large enough (for instance the condition \eqref{eq:conditionTemperature}), we can apply Proposition~\ref{prop:BakryEmery} for times $t\in[0,t_0]$ and then apply the other results to $(m_{t+t_0})_{t\geqslant 0}$.
\end{rem}

The next result addresses the high-diffusivity regime, namely when $\sigma^2$ is high enough (see \eqref{eq:conditionTemperature}). It is proven in Section~\ref{sec:proof_highT}.

\begin{thm}
\label{thm:high_temperature}
Assume that there exist $\rho$, $L$, $R$, $K >0$ such that,
for all $t\geqslant 0$,
\begin{equation}
\label{eq:cond_drift_hightemp}
\bigl(b_{t}(x) - b_{t}(y)\bigr) \cdot (x-y) \leqslant \begin{cases}
- \rho \lvert x-y\rvert^2 &
\forall x,y\in\R^d~\text{with}~\lvert x\rvert\geqslant R\,,\\
L\lvert x-y\lvert^2 & \forall x,y\in\R^d\,,
\end{cases}
\end{equation}
and, setting $R_* = R (2+2L/\rho )^{1/d}$,
\begin{equation}
\label{eq:K}
\sup_{|x|\leqslant R_*} \{ - x\cdot b_{t}(x)\} \leqslant K\,.
\end{equation}
Then, provided $m_0$ satisfies an LSI and
\begin{equation}\label{eq:conditionTemperature}
\sigma^2 \geqslant  \sigma_0^2 \coloneqq
2 (2L+\rho)\frac{(L+\rho/4)R_*^2 + K}{\rho d}\,,
\end{equation}
the family $(m_t)_{t\geqslant 0}$ satisfies a uniform LSI.

Moreover, there exists $C_*>0$ which depends on $L$, $R$, $d$ and $\rho$
but not on $m_0$, $K$ nor $\sigma$ such that,
provided \eqref{eq:conditionTemperature}, there exists $t_*>0$ (depending on all previous parameters, in particular $m_0)$ such that the measure $m_t$ satisfies an LSI
with constant $\sigma^2 C_*$ for all $t\geqslant t_*$.

More precisely, for any $\varepsilon >0$,
there exists $\sigma_0'>0$ which depends only on $L$, $R$, $d$, $\rho$
and $\varepsilon$ such that for all $\sigma \geqslant \sigma_0'$,
there exists $t_*>0$ (depending on all previous parameters, in particular $m_0)$ such that the measure $m_t$ satisfies an LSI
with constant $\sigma^2 (\rho^{-1}+\varepsilon)$ for all $t\geqslant t_*$.
\end{thm}

The restrictive high-diffusivity condition~\eqref{eq:conditionTemperature} is inherited from the method of~\cite{MonmarcheHighTemperature} that we extend here to the time-inhomogeneous setting. It is unclear whether this condition is necessary for this method, based on parallel coupling, to work (even in the time-homogeneous case). A stronger result than in~\cite{MonmarcheHighTemperature} is proven in~\cite{WangExponentialContraction}, based on reflection coupling and without any constraint on $\sigma^2$, but the method crucially relies on a general spectral result  from \cite{MicloHyperboundedness,WangCriteriaSpectralGap}, making the adaptation to the time-inhomogeneous case  (where the law at time $t$ is not given as the invariant measure of some explicit process) much less clear. Notice that, as detailed in Section~\ref{sec:MkV-smooth}, one of our main motivation is to apply the results    of Lacker and Le Flem  in \cite{LLFSharp} which anyway require a sufficiently large $\sigma$ (which is then necessary for the result to hold: at small temperature, counter-examples are known where propagation of chaos cannot be uniform in time).

The next result is the adaptation in the time-inhomogeneous settings of the bounded-Lipschitz perturbation argument of \cite{lsihe}. Its proof is given in Section~\ref{sec:proof_perturbation}.

\begin{thm}
\label{thm:perturbation}
Assume that, for all $t\geqslant 0$, the drift writes $b_t=a_0 + g_t$
for some $a_0$, $g_t\in\mathcal C^1(\R^d,\R^d)$ with bounded derivatives
such that the generator $\mathcal L_0 = a_0\cdot \na + \sigma^2 \Delta $
admits a unique $\mathcal C^2$
invariant probability density $\mu_0$ satisfying an LSI.
Write $\tilde b_t \coloneqq 2\na \ln \mu_0 - b_t$
and $\varphi_t \coloneqq -\na \cdot g_t + g_t \cdot \na \ln\mu_0 $.
Assume that there exist $L$, $R$, $M^\varphi$, $L^\varphi \geqslant 0$
and $\rho > 0$ such that,
\begin{itemize}
\item for all $t\geqslant 0$,
we have $\varphi_t = \varphi_{1,t} + \varphi_{2,t}$
with \(M^\varphi\)-bounded $\varphi_{1,t}$
and $L^\varphi$-Lipschitz $\varphi_{2,t}$;
\item for all $x$, $y\in\R^d$,
\begin{equation}\label{eq:condition_contraction}
\bigl(\tilde b_t(x)-\tilde b_t(y)\bigr) \cdot (x-y) \leqslant \begin{cases}
- \rho |x-y|^2 & \text{if $|x-y|\geqslant R$},\\
L|x-y|^2 & \text{otherwise}.
\end{cases}
\end{equation}
\end{itemize}
Finally, assume that the logarithmic relative density $u_0 = \ln \dd m_0 /\! \dd \mu_0$ exists and is the sum of a bounded and a Lipschitz continuous functions. Then $(m_t)_{t\geqslant 0}$ satisfies a uniform LSI.

Moreover there exists $C_*>0$ which depends on
$L$, $R$, $M^\varphi$, $L^\varphi$, $\sigma^2$, $\rho$
and the LSI constant of $\mu_0$
but not on $m_0$ such that, for some $t_*>0$, $m_t$ satisfies an LSI
with constant $C_*$ for all $t\geqslant t_*$.

Finally, denoting by $C_0$ the LSI constant of $\mu_0$, the following holds.
For any $\varepsilon>0$, there exists $\eta>0$
(which depends only on $\rho$, $L$, $R$ and $\varepsilon$) such that,
if $M^\varphi+L^\varphi \leqslant \eta$,
there exists $t_*>0$ (depending on all previous parameters, in particular $m_0)$ such that $m_t$ satisfies an LSI
with constant $C_0+\varepsilon$ for all $t\geqslant t_*$.
\end{thm}

It may appear that the conditions of Theorem~\ref{thm:perturbation}
are unnatural, in particular for its consideration
of the quantities $\tilde b_t$, $\varphi_t$.
Here we only mention that the proof of the theorem
is based on the Hamilton--Jacobi--Bellman (HJB) equation \eqref{eq:elliptic-hjb}
satisfied by the log-density $\ln \dd m_t/\!\dd \mu_0$ of the measure flow,
where $\tilde b_t$, $\varphi_t$ appear by computations.
This analytical approach,
despite being not very compatible with the theory of Markov diffusion processes,
turns out to be robust enough for our treatment of singular interactions
in Section~\ref{sec:MkV-log-Riesz}.

\section{Proofs of the general results}\label{sec:proofLSI}

In this section we write $(P_{s,t})_{t\geqslant s\geqslant 0}$ the inhomogeneous Markov semi-group associated to \eqref{eq:EDS_inhomogeneous}, given by
\[P_{s,t} f(x) = \mathbb E [ f(X_t) | X_s = x] \,.\]
In particular, the solution of \eqref{eq:FP_inhomogeneous} is then given by $m_t = m_0 P_{0,t}$. In the proofs of Proposition~\ref{prop:BakryEmery} and Theorem~\ref{thm:high_temperature},  we can additionally assume that $b_t$ is smooth with all derivatives being bounded, and consider functions $f$ which are for instance the sum of a positive constant and a compactly supported  smooth non-negative function, which enable to justify the computations based on $\partial_t P_{s,t} f = P_{s,t} \mathcal L_t f $ and  $\partial_s P_{s,t} f = - \mathcal L_s P_{s,t}  f $ (using e.g.\ Proposition~\ref{prop:feynman-kac}). The conclusion is then obtained  by approximation (as in e.g.\ \cite{ulpoc}).

\subsection{Proof of Proposition~\ref{prop:BakryEmery}}\label{sec:proofLSI_BakryEmery}

\begin{proof}[Proof of Proposition~\ref{prop:BakryEmery}]
    Considering $X$ and $X'$ two solutions of \eqref{eq:EDS_inhomogeneous} driven by the same Brownian motion, the condition \eqref{eq:curvature} gives
    \begin{equation}
        \label{eq:ascontraction}
    \dd |X_t - X_t'|^2 \leqslant 2L |X_t-X_t'|^2 \dd t\,,
    \end{equation}
       so that $|X_t-X_t'|^2 \leqslant e^{2L(t-s)}|X_s-X_s'|^2$ for all $t\geqslant s\geqslant 0$, which by \cite{KuwadaDuality} implies
   \begin{equation}
       \label{eq:GradBE}
|\na P_{s,t} f| \leqslant  e^{L (t-s)} P_{s,t}|\na f|\,.
\end{equation}
(In the initial work~\cite{MalrieuLSI}, the sub-commutation~\eqref{eq:GradBE} is obtained by invoking \cite{BakrySobolev}, where it is proven directly through Bakry--Émery interpolations, requiring regularity or integrability assumptions, or approximation arguments to justify the time-differentation along the semi-group. By using a coupling  and relying on the more recent~\cite{KuwadaDuality}, this issue is bypassed. Notice that the fact that a coupling bound implies a sub-commutation is the easiest implication in~\cite{KuwadaDuality}; the converse implication is more involved.)
Fix a function $f\in\mathcal C^\infty(\R^d,\R_+)$, globally Lipschitz continuous and lower bounded by a positive constant (it is sufficient to prove the LSI with these functions and conclude by approximation).
For  $t>s\geqslant 0$, we consider
the interpolation $\Psi(u) = P_{s,u}( P_{u,t} f \ln P_{u,t} f)$ for $u\in[s,t]$,
so that
\begin{align*}
P_{s,t} (f \ln f) - P_{s,t} f \ln P_{s,t} f & = \Psi(t) - \Psi(s)\\
& = \int_s^t \Psi'(u) \dd u \\
&=\sigma^2 \int_{s}^t P_{s,u} \frac{|\na P_{u,t} f|^2}{P_{u,t}f}\dd u  \\
& \leqslant \sigma^2 \int_s^{t} e^{2L(t-u)}\dd u P_{s,t}\biggl(\frac{|\na f|^2}{f}\biggr)\,,
\end{align*}
where we used that $(P_{u,t} |\na f|)^2 \leqslant P_{u,t} (|\na f|^2\!/f) P_{u,t}(f)$ by Cauchy--Schwarz. Integrating with respect to $m_s$ gives
\begin{equation}
    \label{eq:LSIlocalPst}
    m_t(f\ln f)   \leqslant m_s (P_{s,t} f \ln P_{s,t} f)
+ m_t\biggl(\frac{|\na f|^2}{f}\biggr)\sigma^2 \int_s^{t} e^{2L(t-u)}\dd u \,.
\end{equation}
The proof is concluded by applying this with $s=0$ and using the LSI for $m_0$, \eqref{eq:GradBE} and Cauchy--Schwarz to bound
\begin{align*}
m_0( P_{0,t} f \ln P_{0,t} f)
&\leqslant m_t f \ln(m_t f)
+ C_0 m_0\biggl(\frac{|\na P_{0,t} f|^2}{P_{0,t} f}\biggr) \\
&\leqslant m_t f m_t(\ln f)
+ C_0 e^{2Lt} m_t\biggl(\frac{|\na f|^2}{ f}\biggr)\,. \qedhere
\end{align*}
\end{proof}

\subsection{Proof of Theorem~\ref{thm:high_temperature}}\label{sec:proof_highT}
\begin{proof}[Proof of Theorem~\ref{thm:high_temperature}]
    The different steps of the proof are the following. First, using the coupling argument of \cite{MonmarcheHighTemperature} (at high diffusivity), we get a long-time $L^2$ contraction along the synchronous coupling of two solutions of \eqref{eq:EDS_inhomogeneous}. By contrast to the almost sure contraction \eqref{eq:ascontraction}, this $L^2$ contraction is not enough to get an LSI, but it gives a uniform Poincaré inequality following arguments similar to the proof of Proposition~\ref{prop:BakryEmery}. It remains then to prove a so-called defective LSI, which together with the Poincaré inequality yields the desired LSI. The proof of the defective LSI follows the arguments of \cite{lsihe}, except that in the present case the measure for which the LSI is proven is not an invariant measure of a time-homogeneous semi-group (which would solve $\mu = \mu P_t$, which in our case is replaced by $m_t=m_0P_{0,t}$). These arguments combine a Wang--Harnak inequality for the operator $P_{0,t}$ with a Gaussian moment bound.

\proofstep{Step 1: Poincaré inequality}
Let $X$, $X'$ be two solutions of \eqref{eq:EDS_inhomogeneous} driven by the same Brownian motion. Following the proof of \cite[Theorem 1]{MonmarcheHighTemperature} (which is concerned with time-homogeneous processes, but the proof works readily in the non-homogeneous case under the time-uniform assumptions made in Theorem~\ref{thm:high_temperature}), we get for all $t\geqslant s \geqslant 0$,
\begin{equation*}
\mathbb E \bigl[|X_t - X_t'|^2\bigr]
\leqslant M e^{-\lambda (t-s)} \mathbb E \bigl[|X_s - X'_s|^2\bigr]\,,
\end{equation*}
 where
 \begin{equation}
     \label{eq:lambdaM}
\lambda  = \frac \rho 2 \,, \qquad M=    1+ \frac{2(2L+\rho)R_*^2}{4d\sigma^2 } \,.
 \end{equation}
This implies, by \cite{KuwadaDuality},
\[|\na P_{s,t} f|^2 \leqslant M e^{-\lambda (t-s)} P_{s,t}|\na f|^2\,.\]
Since $m_0$ satisfies a LSI, it satisfies a Poincaré inequality, and thus,
\[     m_0(P_{0,t} f)^2   - (m_0P_{0,t}f)^2
\leqslant C_0 m_0 |\na P_{0,t} f|^2 \leqslant C_0 M e^{-\lambda t} m_t|\na f|^2\,. \]
Besides, for fixed $t\geqslant 0$ and  $f\in\mathcal C^1(\R^d,\R)$ globally Lipschitz continuous, considering the interpolation $\Psi(u) = P_{0,u}(P_{u,t}f)^2 $ for $u\in[0,t]$, we get
\begin{align*}
    P_{0,t} (f^2) - (P_{0,t} f)^2 & = \Psi(t) - \Psi(0) \\
    & = \int_0^t \Psi'(u) \dd u\\
    & = \sigma^2  \int_0^t P_{0,u}|\na P_{u,t} f|^2 \dd u  \\
    & \leqslant \sigma^2  \int_0^t M e^{-\lambda(t-u)} \dd u P_{0,t}|\na f|^2\,.
\end{align*}
Combining these last two inequalities, we get
\begin{align}
m_t(f^2) - \bigl(m_t(f)\bigr)^2 &=
m_0\bigl(P_{0,t} (f^2) - (P_{0,t} f)^2\bigr)
+m_0(P_{0,t} f)^2   - (m_0P_{0,t}f)^2 \nonumber  \\
& \leqslant M \biggl(\frac{\sigma^2}{\lambda} + e^{-\lambda t} C_0\biggr)
m_t|\na f|^2\,,\label{eq:Poincare_unif}
\end{align}
which is a uniform Poincaré inequality for $(m_t)_{t\geqslant 0}$.

\proofstep{Step 2: Gaussian moment}
Since $m_0$ satisfies an LSI, there exists $\delta_0>0$ such that
\[\int_{\R^d\times \R^d} e^{\delta_0 |x-y|^2 } m_0(\dd x)m_0(\dd y) < \infty\,. \]
Write $V(x,y) = e^{\delta|x-y|^2}$ for some $0 < \delta < \min(\delta_0,\rho/5)$ and $\mathcal L_{2,t}$ the generator on $\R^d \times \R^d$ of two independent diffusion processes \eqref{eq:EDS_inhomogeneous}, namely
\[\mathcal L_{2,t} g(x,y) = b_t(x) \cdot \na_x + b_t(y)\cdot\na_y + \sigma^2 \Delta_x + \sigma^2 \Delta_y\,.\]
Using \eqref{eq:cond_drift_hightemp} (and that $|x-y| \geqslant 2R$ implies that either $|x|\geqslant R$ or $|y|\geqslant R$),
\begin{align*}
\frac{\mathcal L_{2,t} V(x,y)}{V(x,y)} & = 2\delta (x-y)\cdot\bigl(b_t(x)-b_t(y)\bigr)
+ 4 \delta d + 8 \delta^2 |x-y|^2 \\
& \leqslant \begin{cases}
4 \delta d + (8 \delta^2- 2\delta \rho) |x-y|^2 & \text{if }|x-y|\geqslant 2R\\
4 \delta d + 4(8 \delta^2 + 2\delta L) R^2     & \text{otherwise}
\end{cases}\\
& \leqslant  -\delta \1_{|x-y|\geqslant R_*} + C_*\1_{|x-y|<R_*}
\end{align*}
with
\[R_*^2 = \max\biggl(\frac{1 + 4   d}{2(\rho-4\delta) },4R^2\biggr)\,,\qquad C_* =\delta\bigl(4  d + (2 L+ 8 \delta)R^2\bigr)\,. \]
Hence,
\[
\mathcal L_{2,t}V(x,y) 
\leqslant - \delta  V(x,y) + C_*  e^{\delta R_*^2}\,,\]
and thus,
\[ \partial_t (m_t\otimes m_t) (V)  \leqslant - \delta (m_t\otimes m_t) (V) + C_*  e^{\delta R_*^2}\,. \]
As a conclusion, for all $t\geqslant 0$,
\begin{equation}
    \label{eq:Gaussian_bound}
\int_{\R^d\times\R^d} e^{\delta |x-y|^2} m_t(\dd x)m_t(\dd y)  \leqslant \delta^{-1} C_*  e^{\delta R_*^2} + e^{-\delta t }\int_{\R^d\times\R^d} e^{\delta |x-y|^2} m_0(\dd x)m_0(\dd y) \,.
\end{equation}

\proofstep{Step 3: Wang--Harnack inequality} In the following, fix $f>0$ such that $m_t f=1$.
Using the Röckner--Wang argument~\cite{RoecknerWangLogHarnack} for the diffusion \eqref{eq:EDS_inhomogeneous} we get, for all $x$, $y\in\R^d$, $\alpha>1$ and $t>0$,
\begin{eqnarray}
\label{eq:RocknerWang}
\bigl(P_{0,t} f(y)\bigr)^\alpha  \leqslant (P_{0,t} f^\alpha) (x) \exp \Biggl(\frac{\alpha }{2\sigma^2(\alpha-1)}\biggl(L^2t + \frac{|x-y|^2}{t}\biggr)\Biggr)\,,
\end{eqnarray}
so that
\begin{multline*}
\int_{\R^d} (P_{0,t} f^\alpha) (x) m_0(\dd x) \\ \geqslant
\bigl(P_{0,t} f(y)\bigr)^\alpha \int_{\R^d} \exp \Biggl(- \frac{\alpha }{2\sigma^2(\alpha-1)}\biggl(L^2t + \frac{|x-y|^2}{t}\biggr)\Biggr)m_0(\dd x)\,,
\end{multline*}
and thus, for any $\beta>\alpha$,
\begin{align}
\MoveEqLeft m_0 (P_{0,t} f)^\beta \nonumber \\&\leqslant
\bigl(m_t(f^\alpha)\bigr)^{\beta/\alpha} \int_{\R^d} \biggl[\int_{\R^d} \exp \biggl(- \frac{\alpha \bigl(L^2t + \frac{|x-y|^2}{t}\bigr)}{2\sigma^2(\alpha-1)}\biggr) m_0(\dd x)\biggl]^{-\beta/\alpha}m_0(\dd y) \nonumber \\
& \leqslant \bigl(m_t(f^\alpha)\bigr)^{\beta/\alpha} \int_{\R^{2d}} \exp \biggl(\frac{\beta \bigl(L^2t + \frac{|x-y|^2}{t}\bigr) }{2\sigma^2(\alpha-1)}\biggr)m_0(\dd x)m_0(\dd y)\,. \label{eq:RWbis}
\end{align}
Using Jensen's inequality for the probability with density $P_{0,t} f$ with respect to $m_0$, taking $\alpha=3/2$ so that $x\mapsto x^{\alpha-1}$ is concave, we get
\[ m_t(f^\alpha) = m_0   P_{0,t} f^\alpha
\leqslant \bigl(m_0 (P_{0,t} f)^2\bigr)^{\alpha-1} \,. \]
Using  \eqref{eq:RWbis} with $\beta=2$ to bound the right hand side then gives
\[    m_t (f^\alpha)   \leqslant \bigl(m_t(f^\alpha)\bigr)^{2/3} \Biggl[\int_{\R^{2d}} \exp \Biggl(\frac{4  }{2\sigma^2}\biggl(L^2t + \frac{|x-y|^2}{t}\biggr)\Biggr)m_0(\dd x)m_0(\dd y)\Biggr]^{1/2}\,.\]
and we can divide by $(m_t f^\alpha)^{2/3}$ to end up with
\[    m_t (f^\alpha)   \leqslant\Biggl[ \int_{\R^{2d}} \exp\Biggl(\frac{2 }{\sigma^2}\biggl(L^2t + \frac{|x-y|^2}{t}\biggr)\Biggr)m_0(\dd x)m_0(\dd y)\Biggr]^{3/2} \,.\]
Applying this result to $(m_{t+t_0})_{t\geqslant 0}$ for some $t_0>0$ we get that for all $t\geqslant 0$ and all $f>0$ with $m_{t+t_0} f=1$,
\[m_{t+t_0}\bigl(f^{3/2}\bigr) \leqslant \Biggl[\int_{\R^{2d}} \exp \Biggl(\frac{2 }{\sigma^2}\biggl(L^2t_0 + \frac{|x-y|^2}{t_0}\biggr)\Biggr)m_t(\dd x)m_t(\dd y)\Biggr]^{3/2}\,. \]
Taking $t_0  = 2/(\delta\sigma^2) $, the right hand side is bounded uniformly in $t\geqslant 0$ thanks to \eqref{eq:Gaussian_bound}. As a conclusion, we have determined $t_0$, $C>0$ such that
\begin{equation}
\label{eq:hypercontract}
\forall t \geqslant t_0,\ \forall f>0,\qquad m_t\bigl(f^{3/2}\bigr) \leqslant C (m_t f)^{3/2}\,.
\end{equation}
Moreover, in view of \eqref{eq:Gaussian_bound}, we can find $C>0$ which depends on  $m_0$ only through $\delta$ such that \eqref{eq:hypercontract} holds with this $C$ for all $t$ large enough. To see that we can take $\delta$ independent from $m_0$, we can replace the function $V$ above by the time-dependent $V_t(x,y) = e^{\delta_t|x-y|^2}$ where $t\mapsto \delta_t$ is slowly and smoothly increasing starting from some small $\delta_0>0$ (depending on $m_0$) and reaching $\rho/5$ after some time. Following similar computations as above we get that $(m_t \otimes m_t)(V_t)$ is non-increasing (taking $\dd \delta_t/\dd t$ sufficiently small), from which, replacing $(m_t)_{t\geqslant 0}$ by $(m_{t_0+t})_{t\geqslant 0}$ for some sufficiently large $t_0$, we can assume that \eqref{eq:Gaussian_bound} holds for $\delta=\rho/5$. As a conclusion, for times large enough, \eqref{eq:hypercontract} holds with a constant $C$ independent from $m_0$.

\proofstep{Step 4: Conclusion}
For $t\geqslant t_0$, applying \eqref{eq:LSIlocalPst} with $s=t-t_0$ gives, for $f>0$,
\begin{equation*}
m_t(f\ln f) \leqslant 2 m_s \bigl(P_{s,t} f \ln (P_{s,t} f)^{1/2}\bigr)
+ m_t\biggl(\frac{|\na f|^2}{f}\biggr)\sigma^2 \int_0^{t_0} e^{2Lu}\dd u \,.
\end{equation*}
Assume that $f$ is such that  $m_t f= 1$. Applying Jensen's inequality twice (first with the probability measure with density $P_{s,t} f$ with respect to $m_s$) gives
\[  m_s \bigl(P_{s,t} f \ln (P_{s,t} f)^{1/2}\bigr)
\leqslant \ln \bigl(m_s (P_{s,t} f)^{3/2} \bigr)\leqslant \ln m_t\bigl(f^{3/2}\bigr)\,. \]
Thanks to \eqref{eq:hypercontract}, we have thus obtained that for all $t\geqslant t_0$ and all $f>0$ with $m_t f =1$,
\begin{equation}
    \label{eq:defLSI}
     m_t(f\ln f)   \leqslant 2 \ln C + m_t\biggl(\frac{|\na f|^2}{f}\biggr)\sigma^2 \int_0^{t_0} e^{2Lu}\dd u \,,
\end{equation}
which is called a defective LSI (and is uniform over $t\geqslant t_0$). According to \cite[Proposition 5.1.3]{BGLMarkov}, combining this inequality with the (time uniform) Poincaré inequality \eqref{eq:Poincare_unif} gives an LSI for $m_t$ uniformly over  $t\geqslant t_0$. For $t\in[0,t_0]$ we apply Proposition~\ref{prop:BakryEmery}, which concludes the proof of the uniform LSI.

Finally, as mentioned above, the constant $C$ may be taken independent from  $m_0$, in which case the defective LSI \eqref{eq:defLSI} holds for sufficiently large times. Similarly, we see that the Poincaré inequality \eqref{eq:Poincare_unif} holds with constant  $M \sigma^2\! / \lambda +1$ (which is independent from $m_0$) for $t$ large enough. This shows that there exists $C_*'>0$ independent from $m_0$ such that $m_t$ satisfies an LSI with constant $C_*'$ for $t$ large enough. The fact that  $C_*' \leqslant \sigma^2 C_*$ for some $C_*>0$ independent from $\sigma$ can be checked in the explicit expressions above. More precisely, taking $\delta=\rho/5$ and $t_0 = 2 /(\delta\sigma^2)$, we get that, in \eqref{eq:defLSI} the constant $C$ is uniformly bounded over $\sigma \geqslant \sigma_0$ by a constant that depends only on $\rho$, $L$, $R$, $d$, and similarly we can   bound
\[ \sigma^2 \int_0^{t_0} e^{2Lu}\dd u \leqslant \sigma^2 t_0 e^{2Lt_0} \leqslant \frac{10}{\rho} e^{20L/\rho}   \]
in \eqref{eq:defLSI} uniformly over $\sigma\geqslant 1$. As a consequence, for large values of $\sigma^2$, the leading term in the LSI constant for large times is $\sigma^2 M/\lambda$ from the Poincaré constant, with $M$ and $\lambda$ in \eqref{eq:lambdaM}. As $\sigma \rightarrow \infty$, $M$ goes to $1$, so we may take the LSI constant (for large times) to be $\sigma^2 (\lambda^{-1}+\varepsilon)$ for any arbitrary $\varepsilon>0$ for $\sigma$ large enough. This estimate (with $\lambda=\rho/2$) is not sharp, as we expect an LSI of order $\sigma^2\!/\rho$ (which is the Gaussian behavior). This is due to the $1/2$ factor in the definition of $\lambda$ in \cite{MonmarcheHighTemperature}, which is in fact arbitrary, in the sense that the computations of \cite{MonmarcheHighTemperature} work if we take $\lambda = \alpha \rho$ for an arbitrary $\alpha<1$ (see the two first equations of \cite[Section 2.1.2]{MonmarcheHighTemperature}), provided the lower bound on the temperature $\sigma_0^2$ is sufficiently large (depending on $\alpha$). As a conclusion, we can get a Poincaré constant, and thus an LSI constant, equal to $\sigma^2 (\rho^{-1}+\varepsilon)$ for an arbitrary $\varepsilon$ for large times, provided $\sigma$ is large enough.
\end{proof}

\subsection{Proof of Theorem~\ref{thm:perturbation}}\label{sec:proof_perturbation}

\begin{proof}[Proof of Theorem~\ref{thm:perturbation}]

The proof closely follows the one of \cite[Theorem 1]{lsihe} (in the time-homogeneous settings and with $m_0=\mu_0$, i.e.\ $u_0=0$), the time dependencies appearing along the proof being dealt with the uniform-in-time assumptions of Theorem~\ref{thm:perturbation}. We recall the main arguments and refer to \cite{lsihe} for details. Starting from
\begin{align*}
\partial_t m_t &= -\na\cdot (b_t m_t) + \Delta m_t\,, \\
0 = \partial_t \mu_0 &= -\na\cdot (a_0 \mu_0) + \Delta \mu_0\,,
\end{align*}
    we get that $h_t=m_t/\mu_0$ is a viscosity solution to
    \begin{equation}
\label{eq:elliptic-h}
\partial_t h_t = \Delta h_t + \tilde b_t \cdot \nabla h_t + \varphi_t h_t\,.
\end{equation}
This gives the Feynman--Kac representation
\[
h_t(x) = \Expect \biggl[h_0\bigl(X_t^{t,x}\bigr) \exp\biggl( \int_0^t \varphi_s\bigl(X^{t,x}_s\bigr) \dd s\biggr)\biggr]\,,
\]
where $X^{t,x}$ solves
\[
X^{t,x}_0 = x\,,\qquad
\dd X^{t,x}_s = \tilde b_{t-s} \bigl(X^{t,x}_s\bigr) \dd s + \sqrt 2\dd B_t
\quad\text{for $s \in [0,t]$}.
\]
Suppose additionally that $\varphi_t$,  $h_0$ and $1/h_0$ are bounded and Lipschitz continuous (the general case being obtained afterwards by an approximation argument, which we omit here, referring to \cite{lsihe}).
Then, applying synchronous coupling to the Feynman--Kac formula above, for any $T>0$  we obtain a constant $M > 0$ such that for every $t$, $s \in [0,T]$ and every $x$, $y \in \mathbb R^d$,
\[
M^{-1} \leqslant h(t,x) \leqslant M\qquad\text{and}\qquad
|h(t,x) - h(s,y)| \leqslant M \bigl(|t - s|^{1/2} + |x - y|\bigr)\,.
\]
Taking the logarithm
we obtain that $u_t \coloneqq \ln h_t$ is a bounded and uniformly continuous viscosity solution
to the HJB equation,
\begin{equation} \label{eq:elliptic-hjb}
\partial_t u_t = \Delta u_t + |\nabla u_t|^2 + \tilde b_t\cdot\nabla u_t
+ \varphi_t\,.
\end{equation}
In order to use a stochastic control representation of the solutions of such equations, for $N \in \N$, consider the approximative HJB equation,
\begin{equation} \label{eq:elliptic-hjb-N}
u_0^N = u_0\,,\qquad  \partial_t u^N_t = \Delta u^N_t + \sup_{\alpha:\lvert\alpha\rvert\leqslant N}
\{ 2\alpha\cdot\nabla u^N_t - |\alpha|^2\}
+ \tilde b\cdot \nabla u^N + \varphi_t\,,
\end{equation}
and the associated control problem,
\begin{equation} \label{eq:elliptic-oc-N}
V^N(T, x) = \sup_{\vphantom{\alpha:|\alpha_t| \leqslant N}\nu}
\sup_{\alpha:|\alpha_t| \leqslant N} \mathbb E \biggl[u_0\bigl(X_T^{\alpha,x}\bigr)+ \int_0^T
\Bigl(\varphi_t\bigl(X_t^{\alpha,x}\bigr) - |\alpha_t|^2\Bigr)\dd t \biggr]\,,
\end{equation}
 where $\nu = \bigl(\Omega, F, (\mathcal F_\cdot), \mathbb P, (B_\cdot)\bigr)$
stands for a filter probability space with the usual conditions
and an \((\mathcal F_\cdot)\)-Brownian motion $B$,
$\alpha$ is an $\mathbb R^d$-valued progressively measurable process such that
$\int_0^T \Expect \bigl[|\alpha_t|^m\bigr] \dd t$
is finite for every \(m \in \N\), and $X^{\alpha,x}$ solves
\begin{equation}
    \label{eq:SDEcontrolled}
    X_0^{\alpha,x} = x\,,\qquad
\dd X_t^{\alpha,x} = \Bigl(\tilde b\bigl(X_t^{\alpha,x}\bigr) + 2 \alpha_t\Bigr) \dd t
+ \sqrt{2} \dd B_t \,.
\end{equation}
By Theorem IV.7.1 and the results in Sections V.3 and V.9
of \cite{FlemingSoner}, the value function $V^N$ defined
by \eqref{eq:elliptic-oc-N} is a bounded and uniformly continuous viscosity solution
to \eqref{eq:elliptic-hjb-N}.

Suppose $u_0 = \ln (m_0 / \mu_0)$ is the sum of an
$M^{u_0}$-bounded and an $L^{u_0}$-Lipschitz function.
As shown in \cite[Lemma 8]{lsihe}, using a reflection coupling of two solutions of \eqref{eq:SDEcontrolled} with different initial conditions but using the same control $\alpha$, we get that there exist $C'$, $\kappa  > 0$,
depending only on $\rho$, $L$, $R$,
such that for every $x$, $y \in \R^d$, $N\in\N$, $T > 0$ and   $t >0$,
we have
\begin{multline}
  |V^N(T,x) - V^N(T,y)| \leqslant  2 M^\varphi t + 2 M^{u_0}\1_{T<1}
\\
+ C'\biggl(  \1_{t<T} \frac{M^\varphi}{t}   + L^\varphi + e^{-\kappa T}\bigl(L^{u_0} + \1_{T\geqslant 1} M^{u_0}\bigr)\biggr) |x - y|\,.  \label{eq:VNT}
\end{multline}
We simply take $t=1$. Since both $u$ and $V^N$ are bounded and uniformly continuous on $[0,T]\times\R^d$,
we can apply the parabolic comparison for viscosity solutions on the whole space
\cite[Theorem 1]{DFOParabolicComparison}
to obtain $V^N(T,x) = u_T(x)$ for $N$ sufficiently large.
Therefore, we have obtained that there exists $C>0$ such that for every $T > 0$ and every $x$, $y \in\R^d$,
we have
\begin{equation}
\label{eq:elliptic-bound+lip}
|u_T(x) - u_T(y)| \leqslant C(1+ |x - y|)\,.
\end{equation}
Besides, in view of \eqref{eq:VNT}, we can find $C>0$ independent from $m_0$ such that \eqref{eq:elliptic-bound+lip} holds with this $C$ for all $T$ large enough. Moreover this $C$ can be taken arbitrarily small provided $M^\varphi+L^\varphi$ is small enough.

We can then decompose $u_T$ as the sum of a bounded and a Lipschitz continuous functions (with time uniform bounds for both functions). For instance we can consider a $2C(1+\sqrt{d})$-Lipschitz function $v_T$ that coincides with $u_T$ at all points $x\in\mathbb Z^d$ (thanks to \eqref{eq:elliptic-bound+lip}) and then $u_T - v_T$ is uniformly bounded (thanks to \eqref{eq:elliptic-bound+lip} again) uniformly in $T$. The proof is concluded by applying successively the Holley--Stroock
and Aida--Shigekawa perturbation lemmas \cite{HolleyStroockLSI,AidaShigekawaLSI}.
\end{proof}

\section{Sharp PoC for McKean--Vlasov diffusions}\label{sec:MkV-smooth}

\subsection{Settings and notations}

In this section, we consider the non-linear McKean--Vlasov equation on $\R^d$:
\begin{equation}
    \label{eq:McV}
\partial_t m_t = \nabla \cdot \bigl( \sigma^2 \na m_t - F(\cdot, m_t) m_t\bigr)\,,
\end{equation}
which corresponds to \eqref{eq:FP_inhomogeneous} in the case \eqref{eq:McK-drift}. In fact, since we want to apply the results of \cite{LLFSharp}, we consider its settings, which reads
\[F(x,m) = b_0(x) +  \int_{\R^d} b(x,y)m(\dd y)\]
for some $b_0:\R^d\rightarrow \R^d$ and $b:\R^d\times \R^d\rightarrow \R^d$ (which additionally may depend on time in \cite{LLFSharp}, which we don't consider here for simplicity as it is not the case in the examples were are interested in, although it would work similarly). It is  associated to the system of interacting particles
$\bm X=(X^1,\dots,X^N)$ solving
\begin{equation}
    \label{eq:particles}
\forall i\in\cco 1,N\ccf, \qquad \dd X_t^i = b_0\bigl(X_t^i\bigr)\dd t
+ \frac{1}{N-1} \sum_{j\in\cco 1,N\ccf \setminus\{i\}} b\bigl(X_t^i,X_t^j\bigr)\dd t
+ \sqrt 2\sigma \dd B_t^i\,,\end{equation}
where $B^1$, \dots, $B^N$ are independent $d$-dimensional Brownian motions. Denote by $m_t^N$ the law of $\bigl(X_t^1,\dots,X_t^N\bigr)$ and by $m_t^{k,N}$ the law of $\bigl(X_t^1,\dots,X_t^k\bigr)$ for $k\leqslant N$.

The PoC phenomenon describes the fact that, in the system of interacting particles, as $N\rightarrow\infty$, particles become more and more independent, so that $m_t^{k,N}$ converges to $m_t^{\otimes k}$ for a fixed $k$. Up to recently, known results were typically that,
under suitable conditions, for a fixed $t>0$,
$\bigl\Vert m_t^{k,N}-m_t^{\otimes k}\bigr\Vert_{\TV} = \mathcal O\bigl(\!\sqrt{k/N}\bigr)$.
This can be for instance obtained by showing the global estimate $\mathcal H\bigl(m_t^N\big|m_t^{\otimes N}\bigr) = \mathcal O(1)$ (which is optimal)
using then that $\mathcal H\bigl(m_t^N\big|m_t^{\otimes N}\bigr) = (N/k)\mathcal H\bigl(m_t^{k,N}\big|m_t^{\otimes k}\bigr) $ (assuming for simplicity that $n/k\in\N$) and  concluding with Pinsker's inequality.
This $k/N$ rate for the marginal relative entropy (hence $\sqrt{k/N}$ in TV) was thought to be optimal until Lacker showed in \cite{LackerHier} that it is possible to get a rate $k^2\!/N^2$ by working with a BBGKY hierarchy of entropic bounds instead of  simply with the full entropy of the $N$ particles system. We refer to such entropic estimates with a rate $k^2\!/N^2$ as \emph{sharp} PoC, by comparison with other results (the $k^2\!/N^2$ rate being optimal, as it is reached, e.g., in Gaussian cases). The work \cite{LackerHier} deals with finite-time intervals, and the technique is then refined by Lacker and Le Flem in \cite{LLFSharp} to get uniform-in-time sharp PoC in some cases (small interaction in the torus or convex potentials in $\R^d$). A crucial ingredient in their result is a uniform LSI for the solution of the non-linear equation \eqref{eq:McV}. Our results can thus be applied to extend their results to more general cases, allowing for instance for non-convex potentials on $\R^d$.

The rest of this section is organized as follows. In Section~\ref{subsec:LLF} for the reader's convenience we give a brief overview of the general result of Lacker and Le Flem. In Sections~\ref{subsec:SharpPoCPerturbation}
 and \ref{subsec:SharpPoCHighT} we apply respectively Theorems~\ref{thm:perturbation} and \ref{thm:high_temperature} to get, under suitable conditions, uniform-in-time LSI for solutions of the McKean--Vlasov equation, and thus uniform-in-time sharp PoC as a corollary, in cases which are not covered by \cite{LLFSharp}.

\subsection{Lacker and Le Flem's result}\label{subsec:LLF}

First, for the reader's convenience, we recall \cite[Theorem 2.1]{LLFSharp}.  There are two sets of assumptions to apply this result: Assumption \textbf{E} of \cite{LLFSharp} is technical conditions related to well-posedness  of $m$ and $m^N$ and we omit them as they are not important in our discussion (see Proposition~\ref{prop:LLFAssumptionE} below). The second set of assumptions of \cite{LLFSharp} is the following.

\begin{assu}[Assumption \textbf{A} of \cite{LLFSharp}] The following holds.
\begin{enumerate}
    \item $(m_t)_{t\geqslant 0}$ satisfies a uniform LSI with constant $\eta>0$.
    \item $(m_t)_{t\geqslant 0}$ satisfies a uniform transport inequality: there exists $\gamma>0$ such that, for all $t\geqslant 0$, $x\in\R^d$ and $\nu \in\mathcal P(\R^d)$,
       \begin{equation}
        \label{eq:transportLLF}\bigl|\nu\bigl(b(x,\cdot)\bigr) - m_t\bigl(b(x,\cdot)\bigr)\bigl|^2 \leqslant \gamma \mathcal H(\nu|m_t)\,.
         \end{equation}
    \item $(m_t)_{t\geqslant 0}$ and
$\bigl(m_{t\vphantom{\geqslant 0}}^N\bigr){}^{\vphantom{N}}_{t\geqslant 0}$
satisfy this uniform $L^2$ boundedness:
    \begin{equation}
        \label{eq:L2boundLLF}
        \sup_{N\in\N} \sup_{t\geqslant 0} \int_{\R^{dN}} \bigl| b(x_1,x_2) - m_t\bigl(b(x_1,\cdot)\bigr)\bigr|^2 m_t^N(\dd x) < \infty \,.
    \end{equation}
     \end{enumerate}
\end{assu}

When $b$ is bounded, \eqref{eq:L2boundLLF} is trivial and \eqref{eq:transportLLF} follows from Pinsker's inequality. When $y\mapsto b(x,y)$ is Lipschitz continuous uniformly in $x$, \eqref{eq:L2boundLLF} follows from time-uniform second moment bounds, which are classically obtained by Lyapunov arguments, and \eqref{eq:transportLLF} is implied by the uniform LSI.

\begin{thm}[From Theorem 2.1 of \cite{LLFSharp}]\label{thm:LLF}
    Under Assumptions~\textbf{A} and \textbf{E} of \cite{LLFSharp}, assume moreover that $\sigma^4 > 8 \gamma \eta$ and that
\begin{align}
\exists C_0>0,\ \forall N\geqslant 2,\ \forall k\in \cco 1,N\ccf,\qquad  \label{eq:LLF_initialchaos}
&\mathcal H\bigl(m_0^{k,N}\big| m_0^{\otimes k}\bigr) \leqslant C_0 \frac{k^2}{N^2}\,.
\intertext{Then,}
\label{eq:LLF_chaos}
\exists C>0,\ \forall N\geqslant 2,\ \forall k\in \cco 1,N\ccf,\ \forall t\geqslant 0,\qquad
&\mathcal H\bigl(m_t^{k,N}\big| m_t^{\otimes k}\bigr)\leqslant C \frac{k^2}{N^2}\,.
\end{align}
\end{thm}

\begin{rem}
    \label{rem:LLF}
    As in Remark~\ref{rem:1}, it is in fact sufficient to enforce Assumption \textbf{A} with the condition $\sigma^4 > 8 \gamma \eta$ for times $t\geqslant t_0$ for some $t_0$ and apply Theorem~\ref{thm:LLF} to $(m_{t+t_0})_{t\geqslant 0}$. More precisely, for some $t_0$, assume that \eqref{eq:transportLLF} and \eqref{eq:L2boundLLF} holds uniformly over $t\in[0,t_0]$. Then, assuming the initial chaos \eqref{eq:LLF_initialchaos}, \cite[Theorem 2.2]{LackerHier} gives \eqref{eq:LLF_chaos} for some constant $C>0$ uniformly over $t\in[0,t_0]$. In particular, the initial chaos \eqref{eq:LLF_initialchaos} holds for $(m_{t+t_0})_{t\geqslant 0}$.
\end{rem}

In \cite{LLFSharp}, the assumptions of Theorem~\ref{thm:LLF} are shown to hold in two cases: either convex potentials on $\R^d$, or models on the torus. In any cases, the condition $\sigma^{4} > 8 \gamma \eta$ (corresponding to $r_c>1$ with the notation of \cite[Theorem 2.1]{LLFSharp}) means that the PoC estimates  require that either the temperature $\sigma^2$ is high enough or the strength of the interaction is small enough. In Sections~\ref{subsec:SharpPoCPerturbation} and \ref{subsec:SharpPoCHighT} we extend the range of application of \cite[Theorem 2.1]{LLFSharp} to some cases with non-convex potentials on $\R^d$.

Before that, in order to focus on the uniform LSI afterwards,  let us state a result concerning the other technical conditions, which is sufficient for the cases considered in the two next sections.
\begin{assu}\label{assu:initial_cond}
    The initial conditions $m_0$ and $m_0^N$ have finite moments of all orders, $m_0^N$ is exchangeable and there exists $C$ independent from $N$ such that $\int_{\R^d}|x_1|^2 m_0^{1,N}(\dd x_1) \leqslant C$.
\end{assu}

We omit the proof of the next result, the arguments are the same as in \cite[Corollary 2.7]{LLFSharp}.

\begin{prop}\label{prop:LLFAssumptionE}
Assume that $b_0$ and $b$ are $\mathcal C^1$,
that $|b_0|$ grows at most polynomially,
that $b$ is the sum of a bounded and a Lipchitz continuous function,
and that there exist $c$, $C>0$ such that for all $x$, $y\in\R^d$,
    \[\bigl(b_0(x)+b(x,y)\bigr)\cdot x \leqslant - c |x|^2 + C(1+|y|)\,.\]
     Then, under Assumption~\ref{assu:initial_cond},  $(m_t)_{t\geqslant 0}$
and $\bigl(m_{t\vphantom{\geqslant 0}}^N\bigr){}^{\vphantom{N}}_{t\geqslant 0}$ are well defined and Assumption \textbf{E} of \cite{LLFSharp} and the uniform $L^2$ boundedness \eqref{eq:L2boundLLF} holds.
\end{prop}

\subsection{Convergent  trajectories}\label{subsec:SharpPoCPerturbation}

In this section we focus on the cases where $m_t$ converges as $t\rightarrow\infty$ towards a stationary solution $m_* $ of the non-linear equation \eqref{eq:McV}.
This is known to hold in various cases of interest, like the granular media equation with convex potentials, or repulsive interaction, or high temperature, or small interaction, or other models like the adaptive biasing force method \cite{LRSLongTime} or the  mean-field gradient descent ascent \cite{LuTwoScale}. So, assume that
\begin{equation}
    \label{eq:CVTV}
\|m_t - m_*\|_{\TV} \underset{t\rightarrow \infty}\longrightarrow 0\,.
\end{equation}
\begin{rem}
Under suitable conditions,  \cite[Theorem 4.1]{RenWangEntropy} allows to obtain \eqref{eq:CVTV} from a  $\mathcal W_2$ convergence.
\end{rem}
We now discuss suitable conditions to apply Theorem~\ref{thm:perturbation} with $M^\varphi$, $L^{\varphi}$ arbitrarily small for large times, where we decompose the drift $F(x,m_t) = a_0(x) + g_t(x)$ with $a_0(x) = F(x,m_*)$ and $g_t(x) = F(x,m_t)-F(x,m_*)$. For simplicity we focus on the case where
\begin{equation}
    \label{eq:Fxm}
F(x,m) = - \na V(x) - \int_{\R^d} \na_x W(x,y) m (\dd y)\,,
\end{equation}
for some $V \in\mathcal C^2(\R^d,\R)$ and $W\in\mathcal C^2(\R^{d}\times \R^d,\R)$. The next result would be easily adapted to other cases where the
density of the stationary solutions of \eqref{eq:McV} are explicit or solve an explicit fixed-point equation (namely when the invariant measure of $\sigma^2\Delta + F(\cdot,m)\cdot \na $ is explicit for each $m$), which is for instance the case in \cite{LRSLongTime,LuTwoScale}.

\begin{prop}
    \label{prop:ConvergentMeanField}
    Let $(m_t)_{t\geqslant 0}$ be a solution to \eqref{eq:McV} (in the case \eqref{eq:Fxm}) which converges in TV in long time towards a stationary solution $m_*$. Assume  that $m_0$ admits a density $e^{u_0}$ with respect to $m_*$, with $u_0$ being the sum of a bounded and a Lipschitz continuous function. Assume furthermore that there exists $L$, $\alpha>0$ such that, for all $x$, $y\in\R^d$,
    \begin{equation}
    \label{eq:condWxy}
      |\Delta_x W(x,y)| \leqslant L\,,\qquad  |\na_x W(x,y)| \leqslant \frac{L}{1+|x-y|^\alpha}\,,\qquad |\na V(x)| \leqslant L(1+|x|^\alpha)\,.
    \end{equation}
    Finally, assume that $V$ is strongly convex outside of a compact set. Then, $(m_t)_{t\geqslant0}$ satisfies a uniform LSI. Moreover, as $t\rightarrow \infty$, the optimal LSI constant of $m_t$ converges to  the optimal LSI constant of $m_*$.
\end{prop}

Notice that, $V$ being strongly convex outside a compact set, the last condition of \eqref{eq:condWxy} can only hold with some $\alpha \geqslant 1$. Hence, the second condition of \eqref{eq:condWxy} on $\na W$ means that we only consider local interactions.

\begin{proof}
Considering the decomposition $F(x,m_t) = a_0(x) + g_t(x)$ with $a_0(x) = F(x,m_*)$ and $g_t(x) = F(x,m_t)-F(x,m_*)$,
we have to show that   Theorem~\ref{thm:perturbation} applies to $(m_{t+t_0})_{t\geqslant 0}$  with $M^\varphi$, $L^{\varphi}$ arbitrarily small provided $t_0$ is large enough. Indeed, the last part of Theorem~\ref{thm:perturbation}  will then give that, for any $\varepsilon>0$, the optimal LSI constant of $m_t$ is less than $C_0+\varepsilon$ for $t$ large enough, where $C_0$ is the optimal LSI constant of $m_*$. On the other hand, for any $\varepsilon>0$, there exists a non-constant $\mathcal C^\infty$ function $f$ with compact support such that
\[m_*(f^2 \ln f^2) - m_*(f^2) \ln m_*(f^2) \geqslant (C_0-\varepsilon) m_* |\na f|^2\,.
 \]
 The weak convergence implied by \eqref{eq:CVTV} leads to
 \[m_t(f^2 \ln f^2) - m_t(f^2) \ln m_t(f^2) \geqslant (C_0-2\varepsilon) m_t |\na f|^2
 \]
 for $t$ large enough, which implies that the optimal LSI constant  of $m_t$ is larger than $C_0 -2\varepsilon$.

 Hence, we turn to the application of Theorem~\ref{thm:perturbation} using its notations. We write $m\star W(x) = \int_{\R^d} W(x,y) m(\dd y)$. The invariant measure of $a_0 \cdot \na + \sigma^2 \Delta$ is $\mu_0 = m_*$, with $\na \ln m_* = -\na (V + m_*\star W)=F(\cdot,m_*)$, so that
 \[\tilde b_t(x) = -\na (V + 2 m_*\star W -  m_t \star W) \,.\]
Since $\na_x W$ is bounded by \eqref{eq:condWxy}, the contribution of $W$ in $\tilde b_t$ is bounded (uniformly in $t$) and thus \eqref{eq:condition_contraction}
 holds thanks to the convexity of $V$ outside a compact set.  From \eqref{eq:condWxy},
\[|\na\cdot g_t(x) | = |(m_t-m_*)\star \Delta_x W(x)| \leqslant L \| m_t-m_*\|_{\TV}\,, \]
and, given $(Y,Y')$ an optimal TV coupling of $m_t$ and $m_*$ and using the Cauchy--Schwarz inequality,
\begin{align*}
\MoveEqLeft | g_t(x) \cdot \na \ln m_*(x) |\\
& \leqslant \bigl| \Expect [ \na_x W(x,Y) - \na_x W(x,Y')]\bigr| L (2+|x|^\alpha)    \\
& \leqslant \Expect \biggl[\1_{Y\neq Y'} \biggl(\frac{1}{1+|x-Y|^\alpha} + \frac{1}{1+|x-Y'|^\alpha}\biggr)\biggr] L^2 (2+|x|^\alpha)  \\
& \leqslant \lVert m_t-m_*\rVert_{\TV}^{1/2} \Expect \Biggl[\biggl(\frac{1}{1+|x-Y|^\alpha} + \frac{1}{1+|x-Y'|^\alpha}\biggr)^{\!2}\Biggr]^{\!1/2}  L^2 (2+|x|^\alpha)\,.
\end{align*}
Then we bound, for the first term in the expection,
\begin{align*}
\Expect \biggl[\frac{1}{(1+\lvert x-Y\rvert^\alpha)^2 }\biggr]
&\leqslant \frac{1}{(1+\lvert x/2\rvert^\alpha)^2} + \Proba[|Y| \geqslant |x|/2]
\\&\leqslant \frac{1}{(1+\lvert x/2\rvert^\alpha)^2}
+ \frac{1+\Expect[|Y|^{2\alpha}]}{1+|x/2|^{2\alpha}}\,,
\end{align*}
and similarly for the second term involving $m_*$. Using that $V$ is convex outside a compact set and that $\na_x W$ is bounded we easily get by Lyapunov arguments that the moments of $m_t$ are bounded uniformly in time. As a consequence,
we have obtained,
for $\varphi_t \coloneqq -\na \cdot g_t + g_t \cdot \na \ln\mu_0$, a bound
\[\|\varphi_t\|_\infty \leqslant L' \|m_t-m_*\|_{\TV}^{1/2}  \]
for some $L'$ independent from $t$. The TV convergence \eqref{eq:CVTV} concludes the proof.
\end{proof}

\begin{cor}\label{cor:PoCperturbation}
    Under   Assumption~\ref{assu:initial_cond} and the settings of Proposition~\ref{prop:ConvergentMeanField}, assume furthermore that $W$ is bounded and $V = V_1+V_2$ where $V_1$ is $\rho$-strongly convex and $V_2$ is bounded. Assume that
    \begin{equation}
        \label{eq:condTempCorollary}
         \sigma^2 > \frac{4}\rho \|\na_x W\|_\infty^2 \exp\biggl(\frac{\|V_2\|_\infty+\| W\|_\infty}{\sigma^2}\biggr)\,.
    \end{equation}
    Then, provided the initial PoC \eqref{eq:LLF_initialchaos} holds, so does the uniform in time sharp PoC \eqref{eq:LLF_chaos}.
\end{cor}

This applies to cases on $\R^d$ where $V$ is not convex, which are not covered by \cite{LLFSharp}. In general cases where $V$ may have several local minima, a  condition in the spirit of \eqref{eq:condTempCorollary}, that states that either temperature is large enough or interaction is small enough, is necessary to have a uniform-in-time propagation of chaos estimate.

\begin{proof}
    The assumptions of Proposition~\ref{prop:ConvergentMeanField} imply those of Proposition~\ref{prop:LLFAssumptionE}. Since $\na_x W$ is bounded, Pinsker's inequality gives  the transport inequality \eqref{eq:transportLLF} with $\gamma = \|\na_x W\|_\infty^2/2$.  Proposition~\ref{prop:ConvergentMeanField}  provides the uniform LSI for $(m_t)_{t\geqslant 0}$. Moreover, for large times, the LSI constant of $m_t$ converges to the LSI constant $C_*$ of $m_*$, which by the Bakry--Émery and Holley--Stroock results is less than $\sigma^2 \rho^{-1} \exp \bigl((\|V_2\|_\infty+\| W\|_\infty)/\sigma^{2}\bigr)$. Corollary~\ref{cor:PoCperturbation} thus follows from Theorem~\ref{thm:LLF} (since, as noticed in Remark~\ref{rem:LLF}, the condition $\sigma^4 > 8 \gamma \eta$ only has to be verified for sufficiently long times).
 \end{proof}

\subsection{High temperature regime}\label{subsec:SharpPoCHighT}

Instead of Corollary~\ref{cor:PoCperturbation}, using Theorem~\ref{thm:high_temperature}, we can get an alternative result, which doesn't require the a priori knowledge that $m_t$ converges in large time and with weaker assumptions on $W$, but which only works at high temperature and is less explicit (an explicit condition on $\sigma^2$ can be obtained in principle by checking the proofs, but it wouldn't be as nice as \eqref{eq:condTempCorollary}). In the next statement we consider a solution  $(m_t)_{t\geqslant 0}$ of \eqref{eq:McV} in the case \eqref{eq:Fxm}.

\begin{prop}\label{prop:PoCtemperature}
Under Assumption~\ref{assu:initial_cond}, assume furthemore
that $|\na U|$ grows at most polynomially,
that there exist $\rho$, $L$, $R>0$ such that, for all $z\in\R^d$, $\psi_z:=-\na U - \na_x W(\cdot,z)$ satisfies
     \begin{equation}
        \label{eq:cond_drift_hightemp_2}
     \bigl(\psi_z(x) -\psi_z(y)\bigr) \cdot (x-y) \leqslant \begin{cases}
- \rho |x-y|^2     &  \forall x,y\in\R^d~\text{with}~|x|\geqslant R\,,\\
L|x-y|^2     &  \forall x,y\in\R^d\,,
\end{cases}
\end{equation}
and that $\na_x W = F_1+F_2$ where $F_1$ is bounded and $y\mapsto F_2(x,y)$ is $L_W$-Lipschitz with $8L_W^2 < \rho$, uniformly in $x$. Then, there exists $\sigma_*^2>0$ (which depends on $U$, $W$ and $d$) such that, assuming   $\sigma^2 \geqslant \sigma_*^2$ and  the initial sharp PoC \eqref{eq:LLF_initialchaos}, we have that the uniform in time sharp PoC \eqref{eq:LLF_chaos} holds.
\end{prop}

In particular, if $U$ is strongly convex outside a compact set and $x\mapsto W(x,z)$ is convex for all $z$ with $\na_x W$ being bounded, then Proposition~\ref{prop:PoCtemperature} applies, without any further smallness condition on the interaction. For instance, with  $W(x,z)= a\sqrt{1+|x-z|^2}$, it applies for any $a>0$. However, in that case, the   temperature threshold $\sigma_*^2$ in Proposition~\ref{prop:PoCtemperature} will depend on $a$ and will become large when $a$ is large (i.e.\ when the interaction is strong).

\begin{proof}
We verify the conditions of Theorem~\ref{thm:LLF}.
Using \eqref{eq:cond_drift_hightemp_2} with $y=0$ we see that Proposition~\ref{prop:LLFAssumptionE} holds.
The uniform LSI in the high temperature regime
$\sigma^2 \geqslant \sigma_0^2$ is ensured by Theorem~\ref{thm:high_temperature},
and for times large enough it holds with a constant $\eta = \sigma^2 \eta'$ for some $\eta'>0$ independent from $\sigma$, and which can be taken arbitrarily close to $1/\rho$ for $\sigma^2$ large enough. Here we have used that $\sup\{ - x\cdot b_{t}(x) : |x|\leqslant R_*\}$ can be bounded by a constant $K$ independent from  $t$ and such that \eqref{eq:conditionTemperature} holds for $\sigma$ large enough (for $t$ large enough). Indeed, we can bound
\[|b_t(x)| \leqslant |\na U(x)| + \|F_1\|_\infty + |F_2(x,0)| + L_W \int_{\R^d} |y| m_t(\dd y)\,.\]
Then,  the condition \eqref{eq:cond_drift_hightemp_2} implies that
$s_t\coloneqq\int_{\R^d}|y|^2 m_t(\dd y)$ satisfies
$\dd s_t/\dd t \leqslant - \rho s_t/2 + q + 2d \sigma^2 $
for some $q>0$ independent from $t$ and $\sigma^2$.
From this, for $t$ large enough, we get
$\int_{\R^d} |y| m_t(\dd y) \leqslant C (1+\sigma)$
where $C$ depends only on $d$, $\rho$, $L$, $R$.
As a consequence, in \eqref{eq:K} we can take $K = C'(1+\sigma)$ for some $C'$ (independent from $t$ and $\sigma$), so that \eqref{eq:conditionTemperature} holds for $\sigma$ large enough, as claimed.

It remains to check the transport inequality \eqref{eq:transportLLF}. For any $\theta>0$ we can bound, for all $t\geqslant 0$, $x\in\R^d$ and $\nu \in\mathcal P(\R^d)$,
\begin{align*}
\MoveEqLeft \bigl|\nu\bigl(b(x,\cdot)\bigr) - m_t\bigl(b(x,\cdot)\bigr)\bigr|^2 \\
&\leqslant (1+\theta)\bigl|\nu\bigl(F_1(x,\cdot)\bigr) - m_t\bigl(F_1(x,\cdot)\bigr)\bigr|^2
+ (1+\theta^{-1})\bigl|\nu\bigl(F_2(x,\cdot)\bigr) - m_t\bigl(F_2(x,\cdot)\bigr)\bigr|^2\\
&\leqslant (1+\theta)\|F_1\|_\infty^2 \|\nu - m_t\|_{\TV}^2
+ (1+\theta^{-1})L_W^2 \mathcal W_2^2(\nu, m_t)\\
&\leqslant \gamma \mathcal H(\nu|m_t)\,,
\end{align*}
where we used Pinsker's and Talagrand's inequalities,
and $\gamma$ on the last line is defined by
\[\gamma = \frac{1+\theta}{2}\|F_1\|_\infty^2 + \sigma^2 \eta' (1+\theta^{-1})L_W^2\,.\]
Fixing $\theta$ (independent from $\sigma$) large enough so that $8(1+\theta^{-1})L_W^2 < \rho $, the condition $\sigma^4 > 8\gamma \eta$ holds for $\sigma$ large enough, which concludes.
\end{proof}

\section{Application to log and Riesz interactions}
\label{sec:MkV-log-Riesz}

In this section, we still consider McKean--Vlasov equations \eqref{eq:McV},
but now we impose the following condition on the non-linear drift.

\begin{assu}
\label{assu:MkV-log-Riesz}
We have $d \geqslant 2$, $s \in [0, d - 1)$
and the McKean--Vlasov drift in \eqref{eq:McV} reads
\[
F(x,m) = - \nabla U(x) + M \nabla g_s \star m (x)\,,
\]
where $U$, $M$, $g_s$ satisfy the following conditions:
\begin{itemize}
\item the function $U : \mathbb R^d \to \mathbb R$
has bounded Hessian $\nabla^2 U \in L^\infty$
and satisfies the weak convexity condition:
there exist $\kappa_U > 0$ and $R \geqslant 0$ such that
for all $x$, $y \in \mathbb R^d$ with $\lvert x - y\rvert \geqslant R$, we have
\[
\langle \nabla U(x) - \nabla U(y), x - y\rangle
\geqslant \kappa_U \lvert x - y \rvert^2\,;
\]
\item $g_s : \mathbb R^d \to \mathbb R$ is
the logarithmic or Riesz potential:
\[
g_s(x) = \begin{cases}
- \ln \lvert x \rvert & \text{when}~s = 0, \\
\lvert x \rvert^{-s} & \text{when}~s > 0;
\end{cases}
\]
\item in the sub-Coulombic case where
$s < d - 2$, $M$ is a $d \times d$ real matrix such that
$M : \nabla^2 g(x) \geqslant 0$ for $x \neq 0$;
in the Coulombic and the super-Coulombic cases where $s \in [d - 2, d - 1)$,
$M$ is anti-symmetric.
\end{itemize}
\end{assu}

These models have raised a high interest over the recent years, in particular with a series of work by Rosenzweig, Serfaty and coauthors on the one hand  (see e.g.\ \cite{RSGlobalConvergenceRiesz,CdCRSUniform,RosenzweigSerfatyMLSI} and references within) and Bresch, Jabin, Wang and coauthors on the other hand (see e.g.\ \cite{JabinWang,BJWMFE,BJWAttractive} and references within).
The main result of the section,
to be stated in Theorem~\ref{thm:log-Riesz-quadratic-unif-POC}
in Section~\ref{sec:log-Riesz-quadratic-unif-POC},
addresses the McKean--Vlasov drift force above with $d \geqslant 2$, $s = 0$,
$M$ being anti-symmetric and $U$ being isotropically quadratic.
We show that in this case
the dynamics exhibits the time-uniform propagation of chaos.
This result is a continuation of a recent work of
Guillin, Le Bris and one of the author \cite{GLBMVortex},
where  the uniform PoC is shown for the dynamics
on the torus (thus in a periodic setting).
We also note that a non-time-uniform result on the whole space
have also been obtained very recently by Feng and Wang
\cite{FWVortex}.
In terms of methodology,
the main addition of our work is that
we employ the reflection coupling  technique of Conforti
\cite{ConfortiCouplReflControlDiffusion}
to get regularity bounds for the mean field flow on the whole space
(Theorems~\ref{thm:MkV-log-Riesz-unif-LSI} and
\ref{thm:MkV-log-Riesz-quadratic-convergence}),
which enable to apply the Jabin--Wang method.

We will write $g = g_s$ if that does not lead to ambiguities.
For simplicity, we also set $\sigma = 1$ in this section.
Under the assumptions above, we denote $K = M \nabla g$,
and the McKean--Vlasov dynamics writes
\begin{equation}
\label{eq:MkV-log-Riesz}
\partial_t m_t = \Delta m_t
- \nabla \cdot \bigl( m_t (K \star m_t - \nabla U) \bigr)\,.
\end{equation}
Note that the interaction kernel $K$ is divergence-free
when the matrix $M$ is anti-symmetric.

Consider now the system of $N$ particles in interaction:
\begin{equation}
\label{eq:ps-log-Riesz-SDE}
\dd X^i_t = -\nabla U \bigl( X^i_t \bigr) \dd t
+ \frac 1{N-1} \sum_{j \in \cco 1, N\ccf \setminus \{i\}}
K\bigl(X^i_t - X^j_t\bigr) \dd t
+ \sqrt 2 \dd W^i_t\,, \quad\text{$i = 1$, \dots, $N$,}
\end{equation}
where $W^i_t$ are $N$ independent Brownian motions.
The flow $m^N_t = \Law(\bm X_t) = \Law(X^1_t, \ldots, X^N_t)$
of probabilities in $\mathbb R^{dN}$ satisfies the Fokker--Planck equation
at least formally:
\begin{equation}
\label{eq:ps-log-Riesz}
\partial_t m^N_t
= \sum_{i=1}^N
\Biggl(\Delta_i m^N_t - \nabla_i \cdot
\biggl( \Bigl( \frac 1{N-1}
\sum_{j\in\cco 1, N\ccf \setminus\{i\}} K(x_i - x_j)
- \nabla U(x^i) \Bigr) m^N_t \biggr) \Biggr)\,.
\end{equation}

In this section, $\eta^\varepsilon$ denotes a $\mathcal C^\infty$ mollifier
with support in $B(0,\varepsilon)$ that is also invariant by rotation.
We set $g^\varepsilon \coloneqq g \star \eta^\varepsilon$
and $K^\varepsilon \coloneqq M \nabla g^\varepsilon
= M \nabla g \star \eta^\varepsilon$.
Since under Assumption~\ref{assu:MkV-log-Riesz},
we are restricted to the case where $s < d - 1$,
the interaction potential $g \propto \lvert x\rvert^{-s}$
is integrable around zero,
so $g^\varepsilon$ is infinitely differentiable with bounded derivatives.
Notice that the rotational invariance of $\eta^\varepsilon$ implies that
the value $g^\varepsilon(x)$ depends only on $\lvert x\rvert$
and thus, $\nabla g^\varepsilon(x)$ is parallel to $x$.
 We also work with the approximation of the confinement
$U^\varepsilon \coloneqq U \star \eta^\varepsilon$.

Sometimes, in the rest of this section, for conciseness, we write $A\lesssim B$ when there exists a constant $C$ such that $A\leqslant CB$.

\subsection{Well-posedness of the mean field and particle systems}

For a function $f : \mathbb R^d \to \mathbb R$ and $\theta \in (0,1]$,
we denote the homogeneous $\theta$-Hölder (semi-)norm of $f$ by
\[
[f]_{\mathcal C^\theta}
= \sup_{x,y \in \mathbb R^d\,:\,x\neq y}
\frac{\lvert f(x) - f(y) \rvert}{\lvert x - y\rvert^\theta}\,.
\]
In order to study the singular interaction kernel $K$,
we use the following crucial estimate.
This generalizes the estimate in (2.9) of \cite{RSGlobalConvergenceRiesz}
(which corresponds to the case $p = \infty$).
 We refer readers to Lemma 4.5.4 and Theorem 4.5.10
of \cite{HoermanderAnalysis} for the proof, where the statement of the latter
should be accompanied with an interpolation.

\begin{prop}
\label{prop:log-Riesz-control-L1-Lp}
Let $s > 0$.
For all $m \in L^1 \cap L^p (\mathbb R^d)$
with $\bigl( 1 - \frac sd \bigr)^{-1} < p \leqslant \infty$,
we have
\[
\bigl\lVert \lvert \cdot \rvert^{-s} \star m \bigr\rVert_{L^\infty}
\lesssim \lVert m \rVert_{L^1}^{1 - qs/d}
\lVert m \rVert_{L^p}^{qs/d}\,,
\]
where $p^{-1} + q^{-1} = 1$.
If additionally, for some $\theta \in (0,1)$, we have
$\bigl(1 - \frac{s+\theta}{d}\bigr)^{-1} < p \leqslant \infty$,
then
\[
\bigl[ \lvert \cdot \rvert^{-s} \star m \bigr]_{\mathcal C^\theta}
\lesssim \lVert m \rVert_{L^1}^{1 - q(s+\theta)/d}
\lVert m \rVert_{L^p}^{q(s+\theta)/d}\,.
\]
\end{prop}

Then, we present the well-posedness results for the mean field
and the particle system.

\begin{prop}[Well-posedness of the mean field system]
\label{prop:MkV-log-Riesz-wp-reg}
Let Assumption~\ref{assu:MkV-log-Riesz} hold.
Then we have the following results:
\begin{itemize}
\item For each initial value
$m_0 \in L^1 \cap L^\infty \cap \mathcal P(\mathbb R^d)$,
there exists a unique solution to the mean field flow \eqref{eq:MkV-log-Riesz}
in $C\bigl([0,\infty); L^1(\mathbb R^d) \cap \mathcal P\bigr)
\cap L^\infty\bigl([0,\infty); L^\infty(\mathbb R^d)\bigr)$
depending continuously on the initial value.
In particular, we have the time-uniform bound:
\begin{equation}
\label{eq:MkV-log-Riesz-Linfty-bound}
\sup_{t \in [0, \infty)}
\lVert m_t \rVert_{L^\infty}
\leqslant C_1 (U, \lVert m_0 \rVert_{L^\infty} ) < \infty\,.
\end{equation}
\item If additionally the initial value $m_0$ has finite $k$-th moment
for some $k > 0$, then the mean field flow $m_t$ has finite $k$-th moment,
uniformly in time:
\[
\sup_{t \in [0,\infty)}\int_{\R^d} \lvert x \rvert^k m_t(\dd x)
\leqslant C_2
\biggl(U, K, k, \lVert m_0\rVert_{L^\infty},
\int_{\R^d} \lvert x\rvert^k m_0(\dd x)\biggr)
\]
\item Finally, let $K^\varepsilon = K \star \eta^\varepsilon$,
$U^\varepsilon = U \star \eta^\varepsilon$ be the mollified kernel and
confinement. If $m^\varepsilon_0$ converges to $m_0$ in $L^1$
and if $\sup_\varepsilon \lVert m^\varepsilon_0 \rVert_{L^\infty} < \infty$,
then the solution $m^\varepsilon_t$
of the approximate mean field flow
\begin{equation}
\label{eq:MkV-log-Riesz-approx}
\partial_t m^\varepsilon_t = \Delta m^\varepsilon_t
- \nabla \cdot \bigl( m^\varepsilon_t
(K^\varepsilon \star m^\varepsilon_t - \nabla U^\varepsilon) \bigr)
\end{equation}
converges to $m_t$ in $L^1$ for all $t \geqslant 0$.
Moreover, the $L^\infty$ norm and the $k$-th moment bounds above hold
when we replace $m$ by $m^\varepsilon$.
\end{itemize}
\end{prop}

\begin{prop}[Well-posedness of the particle system]
\label{prop:ps-log-Riesz-wp}
Let Assumption~\ref{assu:MkV-log-Riesz} hold with $s \leqslant d - 2$
and suppose that for all $x \in \mathbb R^d$,
we have $x^\top Mx \leqslant 0$.
Then, for any initial value $\bm X_0 = \bigl( X^1_0, \ldots, X^N_0 \bigr)$
such that $X^i_0 \neq X^j_0$ almost surely for $i \neq j$,
the SDE system \eqref{eq:ps-log-Riesz-SDE} has a global unique strong solution.
Moreover, setting $K^\varepsilon = K \star \eta^\varepsilon$,
$U^\varepsilon = U \star \eta^\varepsilon$,
and considering the approximate SDE system
\begin{equation}
\label{eq:ps-log-Riesz-SDE-approx}
\dd X^{\varepsilon,i}_t
= -\nabla U^\varepsilon\bigl(X^{\varepsilon,i}_t\bigr) \dd t
+ \frac 1{N-1} \sum_{j\in\cco 1,N\ccf\setminus\{i\}}
K^\varepsilon\bigl(X^{\varepsilon,i}_t - X^{\varepsilon,j}_t\bigr) \dd t
+ \sqrt 2 \dd W^i_t\,,
\end{equation}
for $i \in \cco 1, N \ccf$,
with the initial condition $X^{\varepsilon,i}_0 = X^i_0$,
we have, for all $t \geqslant 0$ and $i = 1$, \dots, $N$,
\[
X^{\varepsilon,i}_t \to X^i_t~\text{a.s.,}\qquad
\text{when}~\varepsilon \to 0\,.
\]
\end{prop}

These results may be considered mathematical folklore
and we do not claim originality from them.
Their proofs are postponed to Appendix~\ref{app:log-Riesz-wp}.

\subsection{Uniform Lipschitz and Hessian bounds, and LSI}

We introduce the invariant measure $\mu_0$ of the reversible diffusion
generated by $\Delta - \nabla U \cdot \nabla$,
whose density is explicit:
\[
\mu_0(x) = Z(\mu_0)^{-1} \exp \bigl( - U(x) \bigr)\,,
\quad Z(\mu_0) = \int_{\R^d} \exp \bigl( -U(x) \bigr) \dd x\,.
\]
Note that, under Assumption~\ref{assu:MkV-log-Riesz},
using the HJB flow method of Conforti
(see Theorem 1.3 and Remark 1.7 of \cite{ConfortiWeakSemiconvexity}),
we can show that the measure $\mu_0$ is the image of a Gaussian measure
under a transport mapping with an explicit Lipschitz constant,
and thus satisfies an LSI with an explicit constant.

We use the following result on the Lipschitz and Hessian bounds
on the solution to a class of HJB equations.

\begin{thm}
\label{thm:HJB-global-lip-hess}
Let $T > 0$.
Let $u \in \mathcal C^{1,2}_\textnormal{p}
([0,T] \times \mathbb R^d; \mathbb R)$
be a classical solution to the HJB equation
\[
\partial_t u_t = \Delta u_t - \lvert \nabla u_t \rvert^2
+ \tilde b_t \cdot \nabla u_t + \varphi_t\,,
\]
for some $\tilde b \in \mathcal C^{0,2}_\textnormal{p}
([0,T] \times \mathbb R^d; \mathbb R^d)$,
$\varphi \in \mathcal C^{0,2}_\textnormal{p}
([0,T] \times \mathbb R^d; \mathbb R)$.
Suppose the initial condition $u_0 \in \mathcal C^3_\textnormal{Lip}
(\mathbb R^d; \mathbb R)$.
Suppose the drift $\tilde b$ satisfies the weak convexity condition
\[
\bigl(\tilde b_t(x) - \tilde b_t(y), x - y\bigr)
\leqslant \kappa_{\tilde b}(\lvert x - y\rvert) \lvert x - y\rvert
\]
for some $\mathcal C^1$-continuous
$\kappa_{\tilde b} : (0,\infty) \to \mathbb R$
such that $\int_0^1 r \bigl( \kappa_{\tilde b}(r) \vee 0 \bigr) \dd r < \infty$
and $\liminf_{r \to \infty} \kappa_{\tilde b}(r) < 0$.
Suppose
$\sup_{t \in [0,T]}\lVert \nabla \tilde{b}_t \rVert_{L^\infty} < \infty$.
Then, we have the following results:
\begin{itemize}
\item
If $\varphi_t \in L^\infty$ for all $t \in [0,T]$,
then, we have, for all $t \in [0,T]$,
\begin{equation}
\label{eq:HJB-global-lip}
\lVert \nabla u_t \rVert_{L^\infty}
\leqslant C e^{-c t} \lVert \nabla u_0 \rVert_{L^\infty}
+ \int_0^t \frac{C e^{-cv}}{\sqrt{v \wedge 1}}
\lVert \varphi_{t-v} \rVert_{L^\infty} \dd v\,,
\end{equation}
where $C$, $c > 0$ and depend only on $\kappa_{\tilde b}$.
\item
If additionally, $\nabla \varphi_t \in L^\infty$ for all $t \in [0,T]$,
then we have, for all $t \in [0,T]$,
\begin{multline}
\label{eq:HJB-global-hess}
\lVert \nabla^2 u_t \rVert_{L^\infty}
\leqslant \frac{C'e^{-c't}}{\sqrt{ t \wedge 1}}
\lVert \nabla u_0 \rVert_{L^\infty} \\
+ \int_0^t \frac{C' e^{-c'v}}{\sqrt{v \wedge 1}}
\bigl( \lVert \nabla \varphi_{t-v} \rVert_{L^\infty}
+ \lVert \nabla \tilde b_{t-v} \cdot \nabla u_{t-v} \rVert_{L^\infty} \bigr)
\dd v\,,
\end{multline}
where $C'$, $c' > 0$ and depend only on $\kappa_{\tilde b}$,
$\lVert \nabla u_0 \rVert_{L^\infty}$
and $\sup_{t \in [0,T]} \lVert \varphi_t \rVert_{L^\infty}$.
\end{itemize}
\end{thm}

The theorem is only an enhancement to the result of Conforti
\cite{ConfortiCouplReflControlDiffusion}
by using the short-time gradient estimates obtained
by Priola and Wang \cite{PriolaWangGradientEstimate},
and by Porretta and Priola \cite{PorrettaPriolaGlobalLip}.
Thus we only provide a sketch of proof here.

\begin{proof}[Sketch of proof of Theorem~\ref{thm:HJB-global-lip-hess}]
The stated results differ from the main result of
\cite{ConfortiCouplReflControlDiffusion}, i.e.\ Theorem~1.3 of it,
only in two aspects:
first, we work in a time-non-homogeneous setting;
and second, the uniform gradient estimate that we utilize in the proof has
explosion $t^{-1/2}$ instead of $t^{-1}$ when $t \to 0$.

Following the method in the proof of Theorem~\ref{thm:perturbation}
(and ignoring technical issues
about the correspondance to stochastic control problems),
for every $x$, $y \in \R^d$ and $t \in [0,T]$, we can find
stochastic processes $X^{\alpha,x}_\cdot$, $X^{\alpha,y}_\cdot$, $\alpha$,
all defined on $[0,t]$ and taking values in $\R^d$, such that
\begin{align*}
X^{\alpha,z}_0 = z\,,\qquad \dd X^{\alpha,z}_v
&= \Bigl( \tilde b\bigl(X^{\alpha,z}_v\bigr) + 2\alpha_v \Bigr) \dd v
+ \sqrt 2 \dd B^z_v\,,\qquad\text{for}~z = x,y\,, \\
u(t,x) &= \Expect \biggl[
\int_0^t \Bigl( \lvert \alpha_v\rvert^2
+ \varphi_{t-v}\bigl(X^{\alpha,x}_v\bigr) \biggr) \dd v
+ u\bigl(0, X^{\alpha,x}_t\bigr)\biggr]\,, \\
u(t,y) &\leqslant \Expect \biggl[
\int_0^t \Bigl( \lvert \alpha_v\rvert^2
+ \varphi_{t-v}\bigl(X^{\alpha,y}_v\bigr) \biggr) \dd v
+ u\bigl(0, X^{\alpha,y}_t\bigr)\biggr]\,,
\end{align*}
where $B^x$, $B^y$ are Brownian motions coupled by reflection
until $X^{\alpha,x}$, $X^{\alpha,y}$ collide:
\[
\dd B^{\alpha, y}_v = \biggl( 1 -
\frac{2\bigl(X^{\alpha,y}_v - X^{\alpha,x}_v\bigr)
\bigl(X^{\alpha,y}_v - X^{\alpha,x}_v\bigr)^\mathsf{T}}
{\bigl|X^{\alpha,y}_v - X^{\alpha,x}_v\bigr|^2} \biggr)\dd B^{\alpha, x}_v\,,
\]
for $v \leqslant \tau \coloneqq
\inf \bigl\{ w \geqslant 0 : X^{\alpha,x}_w = X^{\alpha,y}_w\bigr\}$,
and $\dd B^{\alpha,x}_v = \dd B^{\alpha,y}_v$ for $v > \tau$.
Then, by subtracting the dynamics of $X^{\alpha,x}$ and $X^{\alpha,y}$,
we find that their difference process
$\bigl| X^{\alpha,x}_\cdot - X^{\alpha,y}_\cdot \bigr|$
is stochastically dominated by
a one-dimensional It\=o process $(r_t)_{t \geqslant 0}$ solving
\[
\dd r_v = - r_v \kappa_{\tilde b} (r_v) \dd v + 2\sqrt 2\dd W_v
\]
with an absorbing boundary at $0$
with initial value $r_0 = \lvert x - y\rvert$.
It is shown in \cite{PriolaWangGradientEstimate} that
\begin{align*}
\Proba [ r_v > 0 ] &\leqslant \frac{C r_0}{\sqrt{v \wedge 1}}
\intertext{for some $C$ depending only on $\kappa_{\tilde b}$.
Then, by combining the result above with the long-time
Wasserstein contraction studied in \cite{EberleReflectionCoupling}, we get,
for all $v \in [0,t]$,}
\Proba [ r_v > 0 ] &\leqslant \frac{C' e^{-c'v}r_0}{\sqrt{v \wedge 1}}
\end{align*}
for some $C'$, $c' > 0$ depending only on $\kappa_{\tilde b}$.
Therefore, by subtracting the stochastic representation
for $u(t,x)$, $u(t,y)$ and applying the bound above on $r_v$,
we get the first claim.

For the second-order estimate, we take spatial derivatives in the HJB
and find that $\nabla u_t$ solves the $\R^d$-valued equation
\[
\partial_t \nabla u_t
= \Delta \nabla u_t + (\tilde b_t - 2 \nabla u_t) \cdot \nabla^2 u_t
+ \nabla \tilde b_t \cdot \nabla u_t + \nabla \varphi_t\,.
\]
Thus, $\nabla u_t$ solves a second-order equation with
the weakly semi-monotone drift term $\tilde b_t - 2\nabla u_t$
(as $\tilde b_t$ is weakly semi-monotone and $\nabla u_t$ is bounded
by the first claim),
and a bounded source term $\nabla \tilde b_t \cdot \nabla u_t+\nabla \varphi_t$.
Writing the Feynman--Kac formula for $\nabla u_t$
and using the coupling by reflection as above, we get the second claim.
We refer readers to \cite{ConfortiCouplReflControlDiffusion}
for a rigorous justification of the Hessian bound.
\end{proof}

\begin{thm}
\label{thm:MkV-log-Riesz-unif-LSI}
Let Assumption~\ref{assu:MkV-log-Riesz} hold.
Let $m_0 \in \mathcal P(\mathbb R^d)$ be such that
\[
u_0 \coloneqq  - \ln \frac{\dd m_0}{\dd \mu_0}
= - \ln m_0 - U - \ln Z(\mu_0)
\]
is Lipschitz continuous
and let $(m_t)_{t \geqslant 0}$ be the solution to \eqref{eq:MkV-log-Riesz}.
Denote $u_t \coloneqq - \ln \dd m_t /\! \dd \mu_0$.
Then we have, for all $t > 0$,
\[
\sup_{x\in\mathbb R^d}
\bigl|K \star m_t(x) (1+\lvert x\rvert)\bigr| \leqslant C\,,\quad
\lVert \nabla u_t \rVert_{L^\infty} \leqslant C\,,
\quad \lVert \nabla^2 u_t \rVert_{L^\infty} \leqslant
\frac{C}{\sqrt{t \wedge 1}}
\]
for some $C$ depending only on $d$, $s$, $U$, $\lvert M\rvert$,
$\lVert m_0 \rVert_{L^\infty}$ and $\lVert \nabla u_0 \rVert_{L^\infty}$.
Moreover, when $\lvert M\rvert$ increases
and all other dependencies are kept constant, $C$ increases.
Consequently, the flow $(m_t)_{t \geqslant 0}$ satisfies
a uniform LSI whose constant has the same dependency as above
and is increasing in $\lvert M\rvert$.
\end{thm}

The proof of Theorem~\ref{thm:MkV-log-Riesz-unif-LSI}
is postponed to Section~\ref{sec:MkV-log-Riesz-bounds}.

\begin{rem}[Modulated free energy and LSI, and kinetic case]
We remark that since we have obtained
the $L^\infty$ bound of $\nabla^2 u_t$ in the theorem above
(and also in Theorem~\ref{thm:MkV-log-Riesz-quadratic-convergence} below),
we can control the Lipschitz norm of the time-dependent vector field
\begin{equation}
    \label{eq:modulated}
    x \mapsto \sigma^2 \nabla \ln \frac{m_t(x)}{e^{-U(x)}} - K \star m_t(x).
\end{equation}
The control of this quantity, as remarked in \cite[Section 1.2]{CdCRSUniform},
is crucial for the modulated free energy method
since it appears in the ``commutator estimates''.
See e.g.\ \cite[Proposition 1.1]{SerfatyME}
or \cite[Proposition 2.13]{CdCRSUniform}.
We note that unfortunately our method to obtain this control
exploits the long-time contractivity of Brownian motions coupled by reflection,
and relies fundamentally on the diffusivity of the dynamics,
so it is not useful for deterministic dynamics (i.e.\ $\sigma = 0$)
considered originally in \cite{SerfatyME}.
Nevertheless, since similar results for the kinetic case
in the time-homogeneous setting have been established by two of the authors
in \cite[Theorems 2 and 13]{lsihe}, our method provides bounds on $\nabla^2 \ln m_t$,
which is of interest in the perspective of applying the arguments \cite[Theorem 2]{JabinWang} in such hypoelliptic cases.

Besides, together with the control of the Lipschitz norm of \eqref{eq:modulated}, a key ingredient to get uniform-in-time estimates when using modulated free energy instead of relative entropy is the modulated log-Sobolev inequalities discussed in \cite{RosenzweigSerfatyMLSI}. These modulated LSI are in fact classical LSI satisfied uniformly over a specific family of measures (called the modulated Gibbs measures, and distinct from the law $m_t$ that we consider in Theorem~\ref{thm:MkV-log-Riesz-unif-LSI}; but a similar time-uniformity is required). The arguments of the time-uniform LSI of Theorem~\ref{thm:MkV-log-Riesz-unif-LSI} may thus be useful to establish time-uniform modulated LSI (although additional difficulties appear in the latter case, in particular a uniformity in the number of particles is required).  On the topic of modulated free energy and modulated LSI, we   mention that an upcoming work \cite{HRSunpublished} is announced in \cite{CdCRSUniform}.
\end{rem}

\begin{rem}[Non-conservative flow and more singularity]
Two natural extensions to the setting considered
in Assumption~\ref{assu:MkV-log-Riesz}
are to consider a not necessarily anti-symmetric $M$
(notably $M = - I_{d \times d}$ which corresponds to the gradient flow)
and a more singular interaction with $s \in [d - 1, d)$.
We note that in the first case,
a not anti-symmetric $M$ poses challenges in mathematical analysis
since the divergence term
\[
\nabla \cdot (K \star m_t) = M : \nabla^2 g \star m_t
\]
appears and is more singular than the flow $m_t$ itself when
$s \geqslant d - 2$.
By re-examining the proof, we find that
the method of Theorem~\ref{thm:MkV-log-Riesz-unif-LSI} will continue to work
if $(m_t)_{t \geqslant 0}$ satisfies a uniform $\theta$-Hölder bound
for some $\theta > s - d + 3$ without the anti-symmetry of $M$,
or for some $\theta > s - d + 2$ with an anti-symmetric $M$.
The authors are unfortunately unaware of such results for
Riesz flows with confinement in the whole space,
which are possibly worthy of independent studies in the future.
\end{rem}

\subsection{Global PoC for log interaction with general confinement potential}

As a consequence of Theorem~\ref{thm:MkV-log-Riesz-unif-LSI},
we get the strong uniform-in-time propagation of chaos result.

\begin{thm}
\label{thm:log-Riesz-global-POC}
Let Assumption~\ref{assu:MkV-log-Riesz} hold
and suppose additionally that $s = 0$ and $M$ is anti-symmetric.
Let $(m_t)_{t \geqslant 0}$ be a solution
to \eqref{eq:MkV-log-Riesz} whose initial value
$m_0$ satisfies the conditions of Theorem~\ref{thm:MkV-log-Riesz-unif-LSI}
and let $(m^N_t)_{t \geqslant 0}$ be a solution to \eqref{eq:ps-log-Riesz}.
Then, there exist $C$, $C'$, $\rho > 0$, depending only on
$d$, $U$, $\lvert M\rvert$ and $m_0$, such that
\begin{equation}
\label{eq:log-Riesz-global-POC}
\mathcal H \bigl( m^N_t \big| m_t^{\otimes N} \bigr)
\leqslant
C \exp \bigl( - (\rho - C' ) t \bigr)
\mathcal H \bigl( m^N_0 \big| m_0^{\otimes N} \bigr)
+ C \Bigl( 1 + \exp \bigl( - (\rho - C' ) t \bigr) \Bigr)
\end{equation}
for all $t \geqslant 0$, once
$\mathcal H \bigl( m^N_0 \big| m_0^{\otimes N} \bigr) < \infty$.
Moreover, when $\lvert M\rvert$ decreases
and all other dependencies are kept constant,
$C$ decreases, $\rho$ increases,
and $\lim_{\lvert M\rvert \to 0} C' = 0$.
\end{thm}

By the dependency of $C$, $C'$ and $\rho$ in $\lvert M\rvert$,
we find $C' < \rho$
when $\lvert M\rvert$ is sufficiently small,
and in this case
the bound \eqref{eq:log-Riesz-global-POC} becomes uniform in time. Even when $|M|$ is not small, we get a global PoC estimate for the dissipative log-interaction on $\R^d$ with a confinement potential, which is new to our knowledge (the case $U=0$ is addressed in \cite{FWVortex}).

The proof of Theorem~\ref{thm:log-Riesz-global-POC}
is postponed to Section~\ref{sec:log-Riesz-POC}.

\subsection{Uniform PoC for log interaction with quadratic confinement potential}
\label{sec:log-Riesz-quadratic-unif-POC}

In this subsection, we impose the additional assumption.

\begin{assu}
\label{assu:U-quadratic}
The confinement potential reads $U(x) = \kappa_U \lvert x\rvert^2\!/ 2$
for some $\kappa_U > 0$.
\end{assu}

Under Assumptions~\ref{assu:MkV-log-Riesz} and \ref{assu:U-quadratic},
we easily verify that the Gaussian measure $m_*$ with density
\[
m_* (x) = \exp \bigl( -U(x) \bigr)
= \exp \biggl( - \frac{\kappa_U \lvert x \rvert^2}{2} \biggr)
\]
is invariant to the mean field flow \eqref{eq:MkV-log-Riesz}.
The first result that we obtain is the exponential convergence
of the mean field flow towards $m_*$.

\begin{thm}
\label{thm:MkV-log-Riesz-quadratic-convergence}
Let Assumptions~\ref{assu:MkV-log-Riesz} and \ref{assu:U-quadratic} hold
and suppose additionally that $M$ is anti-symmetric.
Let $m_0 \in \mathcal P(\mathbb R^d)$
satisfy the conditions of Theorem~\ref{thm:MkV-log-Riesz-unif-LSI}
and let $(m_t)_{t \geqslant 0}$ be the solution to \eqref{eq:MkV-log-Riesz}.
Then, we have, for all $t \geqslant 0$,
\[
\mathcal H (m_t | m_*) \leqslant \exp(-2\kappa_U t) \mathcal H(m_0 | m_*)\,.
\]
Moreover, setting $u_t = - \ln \dd m_t /\!\dd m_*$,
we have, for all $t > 0$,
\[
\sup_{x\in\mathbb R^d}
\lvert x \cdot K \star (m_t - m_*) \rvert \leqslant C e^{-ct}\,, \quad
\lVert \nabla u \rVert_{L^\infty} \leqslant C e^{-ct}\,,
\quad
\lVert \nabla^2 u \rVert_{L^\infty}
\leqslant \frac{C e^{-ct}}{\sqrt{t\wedge 1}}
\]
for some $C$, $c > 0$ that depend only on
$d$, $s$, $\kappa_U$, $M$ and $\lVert \nabla u_0 \rVert_{L^\infty}$.
\end{thm}

The proof of Theorem~\ref{thm:MkV-log-Riesz-quadratic-convergence}
is postponed to Section~\ref{sec:MkV-log-Riesz-bounds}.

Building upon the exponential convergence above,
we obtain the uniform-in-time propagation of chaos
without restriction on the strength of the interaction.

\begin{thm}
\label{thm:log-Riesz-quadratic-unif-POC}
Let Assumptions~\ref{assu:MkV-log-Riesz} and \ref{assu:U-quadratic} hold
and suppose additionally that $s = 0$ and $M$ is anti-symmetric.
Let $(m_t)_{t \geqslant 0}$ be a solution
to \eqref{eq:MkV-log-Riesz} whose initial value
$m_0$ satisfies the conditions of Theorem~\ref{thm:MkV-log-Riesz-unif-LSI}
and let $(m^N_t)_{t \geqslant 0}$ be a solution to \eqref{eq:ps-log-Riesz}.
Then, there exists $C > 0$, depending only on
$d$, $\kappa_U$, $M$ and $m_0$, such that
\begin{equation}
\label{eq:log-Riesz-quadratic-unif-POC}
\mathcal H \bigl( m^N_t \big| m_t^{\otimes N} \bigr)
\leqslant
C \exp (-2\kappa_U t)
\Bigl( \mathcal H \bigl( m^N_0 \big| m_0^{\otimes N} \bigr) + 1 \Bigr)
\end{equation}
for all $t \geqslant 0$, once
$\mathcal H \bigl( m^N_0 \big| m_0^{\otimes N} \bigr) < \infty$.
\end{thm}

The proof of Theorem~\ref{thm:log-Riesz-quadratic-unif-POC}
is postponed to Section~\ref{sec:log-Riesz-POC}. Notice that, as discussed in e.g.\ \cite{RosenzweigSerfatyMLSI}, this result describes a \emph{generation} of chaos property (not only \emph{propagation}) since it implies that $\mathcal H\bigl(m_t^N\big|m_t^{\otimes N}\bigr)$ is of order $1$ (in terms of $N$) for large times even if it is not the case at time $t=0$.
Here, moreover, and more surprisingly, the right hand side of \eqref{eq:log-Riesz-quadratic-unif-POC} vanishes at $t\rightarrow\infty$, which is due to the fact that in the specific case of an isotropic Gaussian confining potential, the invariant measure of the system of interacting particles is a tensorized Gaussian distribution, which is thus also the long-time limit of the product of solutions of the non-linear equation. Finally, in contrast with the results stated in Section~\ref{sec:MkV-smooth}, here (as in Theorem~\ref{thm:log-Riesz-global-POC}) we only state a result on the relative entropy of the full system, and thus by sub-additivity of the relative entropy this yields PoC estimates on the $k$-particles marginals which are not sharp in the sense of \cite{LackerHier,LLFSharp}.

\subsection{Proofs}

\subsubsection{Proofs of uniform bounds and LSI}
\label{sec:MkV-log-Riesz-bounds}

\begin{proof}[Proof of Theorem~\ref{thm:MkV-log-Riesz-unif-LSI}]
Set $\mu_0 = Z^{-1}\exp(-U)$ with $Z = \int \exp(-U)$.

\proofstep{Step 1: Construction of a regular approximation}
Recall that the initial condition $m_0$ is such that
\[
u_0 = - \ln \frac{\dd m_0}{\dd \mu_0} = - \ln m_0 - \ln Z - U
\]
is Lipschitz continuous.
We construct, for $\varepsilon > 0$, the approximative initial value
\[
m^\varepsilon_0 = \frac{\exp (-u_0\star\eta^\varepsilon) \mu_0}
{\int \exp(- u_0\star\eta^\varepsilon) \mu^\varepsilon_0}\,,
\]
where $\mu^\varepsilon_0 \propto \exp(-U^\varepsilon)$.
Construct as well the approximative dynamics \eqref{eq:MkV-log-Riesz-approx}
with the mollified kernel $K^\varepsilon = K \star \eta^\varepsilon$
and mollified confinement $U^\varepsilon = U \star \eta^\varepsilon$.
By construction, the initial value $m^\varepsilon_0$ converges
to $m_0$ in $L^1$ and is uniformly bounded in $L^\infty$,
thus the last claim of Proposition~\ref{prop:MkV-log-Riesz-wp-reg}
indicates that $m^\varepsilon_t \to m_t$ in $L^1$ for all $t \geqslant 0$.
Using the uniqueness of the solution of the Fokker--Planck equation satisfied by the relative density $\dd m_t^\varepsilon/\!\dd \mu^\varepsilon_0 $
and a Feynman--Kac argument similar to that of Proposition~\ref{prop:feynman-kac},
we obtain that
$u^\varepsilon_t \coloneqq - \ln \dd m_t^\varepsilon / \! \dd \mu^\varepsilon_0$
is $\mathcal C^2$ in space and its derivatives
$\nabla u^\varepsilon_t$, $\nabla^2 u^\varepsilon_t$
grow at most linearly in space (locally in time).
As a consequence, $u^\varepsilon_t$ is a classical solution
to the HJB equation
\[
\partial_t u^\varepsilon_t
= \Delta u^\varepsilon_t - \lvert \nabla u^\varepsilon_t \rvert^2
+ \tilde{b}^\varepsilon_t \cdot \nabla u^\varepsilon_t
+ \varphi^\varepsilon_t\,,
\]
where $\tilde{b}^\varepsilon_t$, $\varphi^\varepsilon_t$ are given by
\begin{align*}
\tilde{b}^\varepsilon_t &= - \nabla U^{\varepsilon}
- K^\varepsilon \star m^\varepsilon_t\,,\\
\varphi^\varepsilon_t &= - \nabla \cdot (K^\varepsilon \star m^\varepsilon_t)
- (K^\varepsilon \star m^\varepsilon_t) \cdot \nabla U^{\varepsilon}\,.
\end{align*}

\proofstep{Step 2: Uniform bound on
$K \star m_t$ and $\nabla u_t$, and uniform LSI}
We verify that the drift $\tilde b^\varepsilon_t$ satisfies
the semi-monotonicity condition of Theorem~\ref{thm:HJB-global-lip-hess},
as the contribution from the interaction $K^\varepsilon \star m^\varepsilon_t$
is controlled by Proposition~\ref{prop:log-Riesz-control-L1-Lp}:
\[
\lVert K^\varepsilon \star m^\varepsilon_t \rVert_{L^\infty}
\leqslant
\lVert K \star m^\varepsilon_t \rVert_{L^\infty}
\lesssim \lVert m^\varepsilon_t \rVert_{L^1}^{1 - (s+1)/d}
\lVert m^\varepsilon_t \rVert_{L^\infty}^{(s+1)/d}\,,
\]
and $U$ (along with its approximation $U^\varepsilon$) is already weakly convex.
Now we focus on proving the uniform $L^\infty$ bound
on $\varphi^\varepsilon_t$.
For the first term in $\varphi^\varepsilon_t$, we find that
in the Coulombic and the super-Coulombic cases,
due to the anti-symmetry of $M$, we have
\[
\nabla \cdot (K^\varepsilon \star m^\varepsilon_t)
= \nabla \cdot (M \nabla g \star m^\varepsilon_t \star \eta^\varepsilon)
= M : g \star \nabla^2 ( m^\varepsilon_t \star \eta^\varepsilon ) = 0\,;
\]
for the sub-Coulombic case where $s < d - 2$,
applying Proposition~\ref{prop:log-Riesz-control-L1-Lp} with $p = \infty$,
we get
\[
\lVert \nabla \cdot (K^\varepsilon \star m^\varepsilon_t) \rVert_{L^\infty}
\lesssim \lVert m^\varepsilon_t \rVert_{L^1}^{1 - (s+2)/d}
\lVert m^\varepsilon_t \rVert_{L^\infty}^{(s+2)/d}\,,
\]
so the first term is uniformly bounded in $L^\infty$ in both cases.
To treat the second term, we note that
\[
\lvert K^\varepsilon \star m^\varepsilon(x) \rvert
\leqslant \sup_{x' \in B(x, \varepsilon)}
\lvert K \star m^\varepsilon(x') \rvert\,,
\]
so it suffices to prove the bound uniformly:
\[
\lvert K \star m^\varepsilon(x') \rvert
\lesssim (1 + \lvert x\rvert)^{-1}\,.
\]
Decompose the kernel in the following way:
\[
K(x) = K(x) \1_{\lvert x\rvert < R}
+ K(x) \1_{\lvert x\rvert \geqslant R}
\eqqcolon K_1(x) + K_2(x)\,.
\]
For the exploding part $K_1$, we have
\begin{align*}
\lvert K_1 \star m^\varepsilon_t (x) \rvert
&= \biggl| \int_{B(x,R)} K(x - y) m^\varepsilon_t(y) \dd y \biggr| \\
&\lesssim
\int_{B(x,R)} \lvert x - y\rvert^{-s-1} m^\varepsilon_t(y) \dd y \\
&\leqslant
\biggl( \int_{B(x,R)} \lvert x - y\rvert^{-p(s+1)} \dd y \biggr)^{\!1/p}
\lVert m^\varepsilon_t \1_{B(x,R)}\rVert_{L^q} \\
&\lesssim R^{d/p - s - 1} \lVert m^\varepsilon_t
\1_{B(x,R)}\rVert_{L^q}\,,
\end{align*}
where $p \in \bigl(1, \frac d{s+1}\bigr)$ and $p^{-1} + q^{-1} = 1$.
For $\lvert x\rvert > R$, we observe
\[
\int_{B(x,R)} (m_t^\varepsilon)^q
\leqslant \lVert m^\varepsilon_t \rVert_{L^\infty}^{q-1}
\int_{B(x,R)} m^\varepsilon_t
\leqslant \lVert m^\varepsilon_t \rVert_{L^\infty}^{q-1}
(\lvert x \rvert - R)^{-q} \int_{\R^d} \lvert x \rvert^q m^\varepsilon_t(x) \dd x\,.
\]
For the non-exploding part $K_2$, we have
\begin{align*}
\lvert K_2 \star m^\varepsilon_t(x) \rvert
&= \biggl| \int_{\mathbb R^d \setminus B(x,R)} K(x - y)
m^\varepsilon_t(y) \dd y \biggr|\\
&\lesssim \int_{\mathbb R^d \setminus B(x,R)}
\lvert x - y \rvert^{-s-1} m^\varepsilon_t(y) \dd y \\
&= \lvert x \rvert^{-s-1} \int_{\mathbb R^d \setminus B(x,R)}
\frac{\lvert x\rvert^{s+1}}{\lvert x - y \lvert^{s+1}}
m^\varepsilon_t(y) \dd y \\
&\lesssim \lvert x \rvert^{-s-1} \int_{\mathbb R^d \setminus B(x,R)}
\frac{\lvert x - y\rvert^{s+1} + \lvert y\rvert^{s+1}}
{\lvert x -y\rvert^{s+1}} m^\varepsilon_t(y) \dd y \\
&\leqslant \lvert x\rvert^{-s-1}
\int_{\R^d} (1 + R^{-s-1} \lvert y\rvert^{s+1}) m^\varepsilon_t(y) \dd y\,.
\end{align*}
Thanks to Proposition~\ref{prop:MkV-log-Riesz-wp-reg},
the mean field flow $(m^\varepsilon_t)_{t \geqslant 0}$
enjoys uniform bounds on its $L^\infty$ norm and its moments,
as all moments of its initial value $m_0^\varepsilon$ are finite.
Thus, we have, uniformly in $t$,
\[
\sup_{t \geqslant 0}
\lvert K \star m^\varepsilon_t (x)\rvert
\leqslant
\sup_{t \geqslant 0}
\lvert K_1 \star m^\varepsilon_t(x)
+ K_2 \star m^\varepsilon_t(x) \rvert
\lesssim (1+\lvert x\rvert)^{-1}\,.
\]
So we have obtained
$\sup_{t\geqslant 0} \lVert \varphi^\varepsilon_t \rVert_{L^\infty} < \infty$,
and the first claim of Theorem~\ref{thm:HJB-global-lip-hess} implies that
$\lVert \nabla u^\varepsilon_t \rVert_{L^\infty}$ is uniformly bounded.
Taking the limit $\varepsilon \to 0$, we recover
the uniform spatial Lipschitz bound on $u_t$ and thus by
the perturbation lemma of Aida--Shigekawa,
the flow $(m_t)_{t \geqslant 0}$ satisfies a uniform LSI.

\proofstep{Step 3: Uniform bound on $\nabla^2 u_t$}
We want to apply the second claim of Theorem~\ref{thm:HJB-global-lip-hess}
to the HJB solution $u^\varepsilon_t$,
and it suffices to control uniformly in time the following quantities:
\begin{align*}
\nabla \tilde b^\varepsilon_t \cdot \nabla u^\varepsilon_t
&= ( - \nabla^2 U^{\varepsilon} - K^\varepsilon \star \nabla m^\varepsilon_t )
\cdot \nabla u^\varepsilon_t\,, \\
\nabla \varphi^\varepsilon_t
&= - \nabla^2 \cdot (K^\varepsilon \star m^\varepsilon_t)
- \nabla (K^\varepsilon \star m^\varepsilon_t) \cdot \nabla U^{\varepsilon}
- (K^\varepsilon \star m^\varepsilon_t) \cdot \nabla^2 U^{\varepsilon}\,.
\end{align*}
The first quantity can be bounded by
\[
\lVert \nabla \tilde b^\varepsilon_t
\cdot \nabla u^\varepsilon_t \rVert_{L^\infty}
\leqslant \lVert \nabla \tilde b^\varepsilon_t \rVert_{L^\infty}
\lVert \nabla u^\varepsilon_t \rVert_{L^\infty}
\leqslant \bigl( \lVert \nabla^2 U \rVert_{L^\infty}
+ \lVert K^\varepsilon \star \nabla m^\varepsilon_t \rVert_{L^\infty} \bigr)
\lVert \nabla u^\varepsilon_t \rVert_{L^\infty}\,,
\]
where $\lVert K^\varepsilon \star \nabla m_t \rVert_{L^\infty}$
is uniformly bounded as
\[
\nabla m^\varepsilon_t = m^\varepsilon_t (- \nabla U^{\varepsilon}
+ \nabla u^\varepsilon_t)
= \frac{\exp(-U^{\varepsilon} - u^\varepsilon_t)}{\int \exp(-U^{\varepsilon} - u^\varepsilon_t)}
(- \nabla U^{\varepsilon} + \nabla u^\varepsilon_t) \in L^1 \cap L^\infty
\]
uniformly in time, thanks to the uniform bound on $\nabla u^\varepsilon_t$.
Now consider the second quantity $\nabla \varphi^\varepsilon_t$.
In the case $s \in [d - 2, d - 1)$,
we have $K = M\nabla g$ with an anti-symmetric $M$,
so the first term $\nabla^2 \cdot (K^\varepsilon \star m^\varepsilon_t)$
vanishes.
In the case $s < d - 2$, we have
\[
\lVert \nabla^2 \cdot (K^\varepsilon \star m^\varepsilon_t) \rVert_{L^\infty}
\leqslant \lVert \nabla K \star \nabla m^\varepsilon_t \rVert_{L^\infty}
\lesssim \lVert \nabla m^\varepsilon_t \rVert_{L^1}^{1 - (s+2)/d}
\lVert \nabla m^\varepsilon_t \rVert_{L^\infty}^{(s+2)/d}\,,
\]
and by the uniform $L^1$ and $L^\infty$ bound on $\nabla m^\varepsilon_t$,
this term is uniformly bounded.
That is to say, in both cases,
the first term $\nabla^2 \cdot (K^\varepsilon \star m^\varepsilon_t)$
is uniformly bounded in $L^\infty$.
As we have $\lVert \nabla^2 U \rVert_{L^\infty} < \infty$,
the third term $(K^\varepsilon \star m^\varepsilon_t) \cdot \nabla^2 U$
is equally uniformly bounded.
So it remains to obtain a uniform bound on the second term
$\nabla (K^\varepsilon \star m^\varepsilon_t) \cdot \nabla U$.
Since $\nabla U^{\varepsilon}$ is of linear growth,
it suffices to prove
\[
\nabla (K \star m^\varepsilon_t)(x)
= (K \star \nabla m^\varepsilon_t)(x)
\lesssim (1 + \lvert x\rvert)^{-1}
\]
uniformly in time.
For this, we use again the decomposition
$K = K_1 + K_2$
in the end of the previous step,
and redoing all the computations,
we find that it is sufficient to uniformly control
\begin{align*}
\int_{\R^d} \lvert x\rvert^q \lvert \nabla m^\varepsilon_t(x) \rvert \dd x
&= \int_{\R^d} \lvert x \rvert^q \lvert - \nabla U(x) + \nabla u^\varepsilon_t(x)\rvert
m^\varepsilon_t (x) \dd x \\
&\lesssim \int_{\R^d} \lvert x\rvert^q (1+\lvert x\rvert) m^\varepsilon_t(x) \dd x
\end{align*}
for some $q > \bigl( 1 - \frac{s+1}{d} \bigr)^{-1}$.
But from Proposition~\ref{prop:MkV-log-Riesz-wp-reg} we know that
the $q$ and $(q+1)$-th moments of $m^\varepsilon_t$ are uniformly bounded.
Hence, $\nabla \varphi^\varepsilon_t$ is uniformly bounded in $L^\infty$ and
by the second claim of Theorem~\ref{thm:HJB-global-lip-hess}, we get that
$\lVert \nabla^2 u^\varepsilon_t \rVert_{L^\infty}$ is uniformly bounded.
Thus $\nabla^2 \ln m^\varepsilon_t = - \nabla^2 U^{\varepsilon} - \nabla^2 u^\varepsilon_t$
is uniformly bounded as well, and taking the limit $\varepsilon \to 0$,
we get the desired result for $u_t$.
\end{proof}

\begin{proof}[Proof of Theorem~\ref{thm:MkV-log-Riesz-quadratic-convergence}]
The proof is similar to that of Theorem~\ref{thm:MkV-log-Riesz-unif-LSI},
except that now the Lipschitz and Hessian bounds converge to zero.
Thus, we first show that the mean field flow $m_t$
converges to the invariant measure $m_*$ and then redo the estimates
on the log-density.

\proofstep{Step 1: Convergence in entropy}
For the initial value $m_0$ such that
\[\nabla u_0 = - \nabla \ln m_0 - \nabla U \in L^\infty,\]
we find, as in the beginning
of the proof of Theorem~\ref{thm:MkV-log-Riesz-unif-LSI},
an approximation $m^\varepsilon_0$ defined by the following:
\[
m^\varepsilon_0 = \frac{\exp(-u_0\star\eta^\varepsilon - U)}
{\int \exp(-u_0\star\eta^\varepsilon - U)}\,.
\]
Set $u^\varepsilon_t \coloneqq - \ln \dd m^\varepsilon_0 / \! \dd m_*$.
We also consider the approximative flow $(m^\varepsilon_t)_{t \geqslant 0}$
solving the mean field Fokker--Planck equation \eqref{eq:MkV-log-Riesz-approx}.
 Notice that, in the case of quadratic potential,
the mollified potential $U^\varepsilon = U \star \eta^\varepsilon$
is nothing but $U$ translated by a constant,
due to the symmetry of $\eta^\varepsilon$.
By Feynman--Kac arguments, we get that $m^\varepsilon_t$
is a classical solution to the Fokker--Planck
and $\nabla^i u^\varepsilon_t$ grows at most linearly for $i = 0$, $1$, $2$.
Thus, we can derive $t \mapsto \mathcal H(m^\varepsilon_t | m_*)$ and get
\begin{align*}
\frac{\dd \mathcal H(m^\varepsilon_t | m_*)}{\dd t}
&= - \mathcal I(m^\varepsilon_t | m_*)
+ \int_{\R^d} \nabla \ln \frac{m^\varepsilon_t(x)}{m_*(x)} \cdot
K^\varepsilon \star (m^\varepsilon_t - m_*) (x) m^\varepsilon_t(\dd x) \\
&= - \mathcal I(m^\varepsilon_t | m_*) - \int_{\R^d} \nabla \ln m_*(x) \cdot
K^\varepsilon \star (m^\varepsilon_t - m_*) (x) m^\varepsilon_t(\dd x) \\
&= - \mathcal I(m^\varepsilon_t | m_*) - \int_{\R^d} \nabla \ln m_*(x) \cdot
K^\varepsilon \star m^\varepsilon_t (x) m^\varepsilon_t(\dd x) \\
&= - \mathcal I(m^\varepsilon_t | m_*)
+ \kappa_U \int x \cdot K^\varepsilon \star m^\varepsilon_t(x) m_t(\dd x) \\
&= - \mathcal I(m^\varepsilon_t | m_*)
+ \kappa_U \iint_{\R^d\times\R^d} x \cdot K^\varepsilon(x-y)
m^\varepsilon_t(\dd x) m^\varepsilon_t(\dd y) \\
&= - \mathcal I(m^\varepsilon_t | m_*)
+ \frac{\kappa_U}{2} \iint_{\R^d \times\R^d} (x - y) \cdot K^\varepsilon(x-y)
m^\varepsilon_t(\dd x) m^\varepsilon_t(\dd y) \\
&= - \mathcal I(m^\varepsilon_t | m_*) \leqslant
- 2\kappa_U \mathcal H(m^\varepsilon_t | m_*)\,.
\end{align*}
Here the second inequality is due to the integration by parts
and the fact that $\nabla \cdot K^\varepsilon = 0$;
the third to the fact that $\nabla \ln m_*(x)$ is parallel to $x$
and $K^\varepsilon \star m_*(x) = K \star (m_* \star \eta^\varepsilon)(x)$
is always orthogonal to $x$,
as $m_* \star \eta^\varepsilon$ is invariant by rotation;
the sixth due to the oddness of $K^\varepsilon$;
and the last due to
\[
x \cdot K^\varepsilon (x)
= x^\top M \nabla g^\varepsilon(x)
\]
and $\nabla g^\varepsilon(x)$ is parallel to $x$.
Then applying Grönwall's lemma and the log-Sobolev inequality for $m_*$, we get
\[
\mathcal H (m^\varepsilon_t | m_*)
\leqslant \mathcal H(m^\varepsilon_0 | m_*) \exp ( - 2\kappa_U t )\,,
\]
and taking the limit $\varepsilon \to 0$ and using the lower semi-continuity
of relative entropy, we recover the first claim.

\proofstep{Step 2: Decaying bound on $x \cdot K \star m_t(x)$
and $\nabla u_t$}
In the following, $C$, $c$ will denote positive reals
that has the same dependency as stated in the theorem
and may change from line to line.
Working again with the approximation $m^\varepsilon_t$,
we get by Pinsker's inequality,
\begin{align*}
\lVert m^\varepsilon_t - m_* \rVert_{L^1}
&\leqslant \exp ( - \kappa_U t ) \sqrt{2 \mathcal H(m^\varepsilon_0 | m_*)} \\
&\leqslant \exp( - \kappa_U t )
\sqrt{ 2\kappa_U^{-1} \mathcal I(m^\varepsilon_0 | m_*) } \\
&\leqslant \exp( - \kappa_U t ) \sqrt{ 2\kappa_U^{-1}
\lVert \nabla u^\varepsilon_0 \rVert_{L^\infty}^2} = C \exp( -\kappa_U t)\,.
\end{align*}
Then, applying Proposition~\ref{prop:log-Riesz-control-L1-Lp}, we get
\[
\lVert K^\varepsilon \star (m^\varepsilon_t - m_*) \rVert_{L^\infty}
\leqslant C \lVert m^\varepsilon_t - m_* \rVert_{L^1}%
^{1 - (s+1)/d}
\lVert m^\varepsilon_t - m_* \rVert_{L^\infty}^{(s+1)/d}
\leqslant C e^{-ct}\,.
\]
We know that $u^\varepsilon_t = - \ln \dd m^\varepsilon_t /\! \dd m_*$
solves the HJB equation
\[
\partial_t u^\varepsilon_t = \Delta u^\varepsilon_t
- \lvert \nabla u^\varepsilon_t \rvert^2
+ \tilde b^\varepsilon_t \cdot \nabla u_t + \varphi^\varepsilon_t
\]
for $\tilde b^\varepsilon_t(x)
= - \kappa_U x - K^\varepsilon \star m^\varepsilon_t(x)$
and $\varphi^\varepsilon_t (x)
= - \kappa_U x \cdot K^\varepsilon \star m^\varepsilon_t(x)$.
Note that $\varphi_t$ satisfies
\[
\varphi^\varepsilon_t (x)
= - \kappa_U x \cdot K^\varepsilon
\star (m^\varepsilon_t - m_*) (x)\,,
\]
since $x \cdot K^\varepsilon \star m_*(x) = 0$ according to the argument
in Step~1.
Thus, we have
\begin{align*}
\varphi^\varepsilon_t(x)
&= - \kappa_U \int_{\R^d} x^\top M \nabla g^\varepsilon(x - y)
(m^\varepsilon_t - m_*)(\dd y) \\
&= - \kappa_U \int_{\R^d} y^\top M \nabla g^\varepsilon(x - y)
(m^\varepsilon_t - m_*)(\dd y)\,,
\end{align*}
as $x^\top M \nabla g^\varepsilon(x) = 0$ for all $x \in \mathbb R^d$.
So $\varphi^\varepsilon_t$ satisfies the bound
\begin{align*}
\lVert \varphi^\varepsilon_t \rVert_{L^\infty}
&= \biggl|
\int_{\R^d} y^\top K^\varepsilon(x - y)
(m^\varepsilon_t - m_*) (\dd y) \biggr| \\
&\lesssim
\int_{B(0,1)} \frac{\lvert y\rvert}{\lvert x - y\rvert^{s+1}}
\lvert m^\varepsilon_t - m_*\rvert (\dd y)
+ \sup_{y : \lvert y-x\rvert \geqslant 1}
\lvert y^\top K^\varepsilon(x - y) \rvert\,
\lVert m^\varepsilon_t - m_* \rVert_{L^1} \\
&\lesssim
\lVert (m^\varepsilon_t - m_*) \1_{B(x,1)}\rVert_{L^q}
+ \lVert m^\varepsilon_t - m_* \rVert_{L^1}
\end{align*}
for $q > \bigl(1 - \frac{s+1}{d}\bigr)^{-1}$, according
to the argument in the proof of Theorem~\ref{thm:MkV-log-Riesz-unif-LSI}.
For the $L^q$ norm we have, by interpolation,
\[
\lVert (m^\varepsilon_t - m_*) \1_{B(x,1)}\rVert_{L^q}
\leqslant
\lVert m^\varepsilon_t - m_*\rVert_{L^q}
\leqslant
\lVert m^\varepsilon_t - m_*\rVert_{L^1}^{1/q}
\lVert m^\varepsilon_t - m_*\rVert_{L^\infty}^{1/p}
\lesssim
\lVert m^\varepsilon_t - m_*\rVert_{L^1}^{1/q}
\]
for $p^{-1}+q^{-1} = 1$.
Thus, $\lVert \varphi^\varepsilon_t \rVert_{L^\infty} \leqslant C e^{-ct}$,
and applying the first claim of Theorem~\ref{thm:HJB-global-lip-hess}, we get
$\lVert \nabla u^\varepsilon_t \rVert_{L^\infty} \leqslant C e^{-ct}$.
The first claim is then proved by taking the limit $\varepsilon \to 0$.

\proofstep{Step 3: Decaying bound on $\nabla^2 u_t$}
First, we have
\[
\lVert \nabla \tilde b^\varepsilon_t
\cdot \nabla u^\varepsilon_t \rVert_{L^\infty}
\leqslant
\lVert \nabla \tilde b^\varepsilon_t \rVert_{L^\infty} \cdot C e^{-ct}
\leqslant
\bigl(\lVert \nabla^2 U\rVert_{L^\infty}
+ \lVert K^\varepsilon \star \nabla m^\varepsilon_t \rVert_{L^\infty}\bigr)
\cdot C e^{-ct}\,,
\]
where
\[
\lVert K^\varepsilon \star \nabla m^\varepsilon_t \rVert_{L^\infty}
\lesssim \lVert \nabla m^\varepsilon_t \rVert_{L^1}^{1 - (s+1)/d}
\lVert \nabla m^\varepsilon_t \rVert_{L^\infty}^{(s+1)/d}\,.
\]
As we have
\[
\nabla m^\varepsilon_t
= - (\nabla U + \nabla u^\varepsilon_t) m^\varepsilon_t
=   - \frac{\nabla U \exp ( - U - u^\varepsilon_t) }
{\int \exp ( - U - u^\varepsilon_t)} - \nabla u^\varepsilon_t\, m^\varepsilon_t
\,
\]
with $\nabla U$ of linear growth and $\nabla u^\varepsilon_t$
being uniformly bounded,
we find that $\nabla m^\varepsilon_t \in L^1 \cap L^\infty$ uniformly.
Thus,
\[
\lVert \nabla \tilde b^\varepsilon_t
\cdot \nabla u^\varepsilon_t \rVert_{L^\infty}
\leqslant C e^{-ct}\,.
\]
The gradient of $\varphi^\varepsilon_t$ reads
\begin{align*}
\nabla \varphi^\varepsilon_t(x)
&= - \nabla \bigl ( \kappa_U x \cdot K^\varepsilon
\star (m^\varepsilon_t - m_*) (x) \bigr) \\
&= - \kappa_U K^\varepsilon \star (m^\varepsilon_t - m_*) (x)
- \kappa_U x \cdot K^\varepsilon \star \nabla (m^\varepsilon_t - m_*) (x)\,.
\end{align*}
The first term on the right hand side is already controlled:
\[
\lvert K^\varepsilon \star (m_t - m_*) (x) \rvert
\lesssim \lVert m_t - m_* \rVert_{L^1}^{1 - (s+1)/d}
\lVert m_t - m_* \rVert_{L^\infty}^{(s+1)/d}
\leqslant C e^{-ct}\,,
\]
and in the following we show that the same is true for the second term.
Again, using the fact that $x \cdot K^\varepsilon(x) = 0$, we get
\begin{align*}
x \cdot K^\varepsilon \star \nabla (m^\varepsilon_t - m_*) (x)
&= \int_{\mathbb R^d} x^\top M \nabla g^\varepsilon(x - y)
\nabla (m^\varepsilon_t - m_*) (\dd y) \\
&= \int_{\mathbb R^d} y^\top M \nabla g^\varepsilon(x - y)
\nabla (m^\varepsilon_t - m_*) (\dd y)\,.
\end{align*}
Following the argument in Step~2, we separate the two cases
$\lvert y - x\rvert < 1$ and $\geqslant 1$, and get
\[
\sup_{x \in \mathbb R^d} \lvert x \cdot K^\varepsilon
\star \nabla (m^\varepsilon_t - m_*) (x)\rvert
\lesssim \lVert \nabla (m^\varepsilon_t - m_*) \rVert_{L^q}
+ \lVert \nabla (m^\varepsilon_t - m_*) \rVert_{L^1}
\]
for $q > \bigl( 1 - \frac{s+1}{d} \bigr)^{-1}$.
Using the explicit density of $m^\varepsilon_t$, we get
\[
\nabla m^\varepsilon_t - \nabla m_* =
- \nabla u^\varepsilon_t\,m^\varepsilon_t
- \nabla U (m^\varepsilon_t - m_*)\,.
\]
The first term satisfies
\[
\lVert \nabla u^\varepsilon_t\,m^\varepsilon_t \rVert_{L^1}
\leqslant \lVert \nabla u^\varepsilon_t \rVert_{L^\infty}
\lVert m^\varepsilon_t \rVert_{L^1}
\leqslant C e^{-ct}\,,
\]
and the second satisfies
\[
\lVert \nabla U (m^\varepsilon_t - m_*) \rVert_{L^1}
\leqslant \lVert \nabla^2 U \rVert_{L^\infty}
W_1 (m^\varepsilon_t, m_*)
\lesssim \sqrt{\mathcal H(m^\varepsilon_t | m_*)} \leqslant C e^{-ct}\,.
\]
Finally, their densities have the $L^\infty$ bounds:
\begin{align*}
\lVert \nabla u^\varepsilon_t\,m^\varepsilon_t \rVert_{L^\infty}
&\leqslant \lVert \nabla u^\varepsilon_t \rVert_{L^\infty}
\lVert m^\varepsilon_t \rVert_{L^\infty} \leqslant C\,, \\
\lVert \nabla U (m^\varepsilon_t - m_*) \rVert_{L^\infty}
&\lesssim \sup_{x\in\mathbb R^d}
(1+\lvert x\rvert) \biggl(
\frac{\exp\bigl(-U(x)\bigr)}{\int \exp(-U)}
+ \frac{\exp\bigl(-u^\varepsilon_t(x)-U(x)\bigr)}
{\int \exp(-u^\varepsilon_t-U)} \biggr)
\leqslant C\,.
\end{align*}
Then, by the same interpolation as in Step~2, we get
$\lVert \nabla \varphi^\varepsilon_t \rVert_{L^\infty} \leqslant Ce^{-ct}$.
Applying the second claim of Theorem~\ref{thm:HJB-global-lip-hess}, we get
\[
\lVert \nabla^2 u^\varepsilon_t \rVert_{L^\infty}
\leqslant
\frac{C e^{-ct}}{\sqrt {t\wedge 1}}
+ \int_0^t \frac{C e^{-cv}}{\sqrt{v \wedge 1}} \cdot C e^{-c(t-v)} \dd v
\leqslant
\frac{C e^{-ct}}{\sqrt {t\wedge 1}}\,.
\]
Taking the limit $\varepsilon \to 0$, we recover the second claim
and this concludes the proof.
\end{proof}

\subsubsection{Proofs of propagation of chaos}
\label{sec:log-Riesz-POC}

\begin{proof}[Proof of Theorem~\ref{thm:log-Riesz-global-POC}]
According to Propositions~\ref{prop:MkV-log-Riesz-wp-reg}
and \ref{prop:ps-log-Riesz-wp}, given the initial values
$m_0$, $m^N_0$, we find respectively approximating sequences
$m^\varepsilon_0$, $m^{\varepsilon, N}_0$ such that
$\ln m^\varepsilon_0 + U^{\varepsilon} \in \mathcal C^\infty_\textnormal{b}$
and $\bm x\mapsto
\ln m^{\varepsilon, N}_0 (\bm x) + \sum_{i=1}^N U^{\varepsilon}(x^i)
\in \mathcal C^\infty_\textnormal{b}$. The solutions of \eqref{eq:MkV-log-Riesz-approx} and of the forward Kolmogorov equation associated to \eqref{eq:ps-log-Riesz-SDE-approx} being unique, we can use the Feynman--Kac representation of Proposition~\ref{prop:feynman-kac} to find that the densities and their classically derivatives
\[
\nabla^i\bigl(\ln m^\varepsilon_t + U^{\varepsilon}\bigr)\,,\quad
\nabla^i\biggl(\ln m^{\varepsilon, N}_t (\bm x)
+ \sum_{i=1}^N U^{\varepsilon}(x^i)\biggr)\,,\qquad i\geqslant 1
\]
exist and grow at most linearly in space (locally in time).
Then in the following we can justify
all the exchanges between limit and integration,
and all the integrations by parts.
Taking the derivative of the relative entropy
$\mathcal H^\varepsilon_t
= \mathcal H \bigl( m^{\varepsilon,N}_t
\big| (m^\varepsilon_t)^{\otimes N} \bigr)$,
and denoting the relative Fisher information by
\[
\mathcal I^\varepsilon_t
= \mathcal I \bigl( m^{\varepsilon,N}_t
\big| (m^\varepsilon_t)^{\otimes N} \bigr)
= \int_{\mathbb R^{dN}} \biggl|\nabla
\ln \frac{m^{\varepsilon,N}_t(\bm x)}
{(m^\varepsilon_t)^{\otimes N}(\bm x)} \biggr|^2
m^{\varepsilon,N}_t(\dd \bm x)\,,
\]
we get
\begin{align*}
\frac{\dd \mathcal H^\varepsilon_t}{\dd t}
&= - \mathcal I^\varepsilon_t + \frac 1{N-1} \sum_{i\neq j}
\int_{\mathbb R^{dN}}
\nabla_i \ln \frac{m^{\varepsilon,N}_t(\bm x)} {m^\varepsilon_t(x^i)} \\
&\hphantom{= - \mathcal I^\varepsilon_t
+ \frac 1{N-1} \sum_{i\neq j} \int_{\mathbb R^{dN}}}
\quad\cdot \bigl(
K^\varepsilon(x^i - x^j) -K^\varepsilon \star m^\varepsilon_t(x^i) \bigr)
m^{\varepsilon,N}_t (\dd \bm x) \\
&= - \mathcal I^\varepsilon_t - \frac 1{N-1} \sum_{i\neq j}
\int_{\mathbb R^{dN}} \nabla \ln m^\varepsilon_t(x^i) \\
&\hphantom{= - \mathcal I^\varepsilon_t - \frac 1{N-1} \sum_{i\neq j}
\int_{\mathbb R^{dN}}}
\quad \cdot
\bigl( K^\varepsilon(x^i - x^j) -K^\varepsilon \star m^\varepsilon_t(x^i) \bigr)
m^{\varepsilon,N}_t (\dd \bm x)\,,
\end{align*}
where $i$, $j$ are summed over $\cco 1,N \ccf$
and the second equality is due to integration by parts
and the fact that $\nabla \cdot K^\varepsilon = 0$.
Noting that the regularized $N$-particle measure $m^{\varepsilon, N}_t$
has density and has no mass on the sets $\{\bm x : x^i = x^j\}$
for $i \neq j$, we find that the second term is equal to, by symmetrization,
\[
- \frac 1{N-1} \sum_{i,j = 1}^N \int_{\mathbb R^{dN}} \phi_t(x^i, x^j)
m^{\varepsilon,N}_t (\dd \bm x)\,,
\]
where the function $\phi_t(\cdot, \cdot)$ is given by
\begin{multline}
\label{eq:def-phi}
\phi_t(x, y)
= \frac 12 K^\varepsilon(x - y)
\cdot \bigl( \nabla \ln m^\varepsilon_t(x)
- \nabla \ln m^\varepsilon_t(y) \bigr) \1_{x \neq y}
\\ - \frac 12 K^\varepsilon \star m^\varepsilon_t(x)
\cdot \nabla \ln m^\varepsilon_t(x)
- \frac 12 K^\varepsilon \star m^\varepsilon_t(y)
\cdot \nabla \ln m^\varepsilon_t(y)\,.
\end{multline}
The function $\phi_t$ satisfies
\begin{align*}
\int_{\mathbb R^d} \phi_t (x, y) m^\varepsilon_t (\dd y) &= 0\,, \\
\int_{\mathbb R^d} \phi_t (y, x) m^\varepsilon_t (\dd y) &= 0
\end{align*}
for all $x \in \mathbb R^d$.
From now on, the symbols $C_i$, $i \in \mathbb N$ denote
a positive number that has the same dependency as $C$, $\rho$ have
in the statement of the theorem.
For the first term in \eqref{eq:def-phi},
we have by Theorem~\ref{thm:MkV-log-Riesz-unif-LSI},
\[
\sup_{x,y\,:\,x\neq y}
\bigl|K^\varepsilon(x - y) \cdot \bigl( \nabla \ln m^\varepsilon_t(x)
- \nabla \ln m^\varepsilon_t(y) \bigr)\bigr|
\leqslant C_1 \lvert M\rvert
\lVert \nabla^2 \ln m_t \rVert_{L^\infty}
\leqslant \frac{C_2 \lvert M\rvert}{\sqrt {t \wedge 1}} \,.
\]
For the last two terms in the definition \eqref{eq:def-phi} of $\phi_t$,
we have by the same theorem,
\begin{align*}
\lvert K \star m^\varepsilon_t(x)\rvert
&\leqslant C_3 \lvert M\rvert (1 + \lvert x\rvert)^{-1}\,, \\
\lvert \nabla \ln m^\varepsilon_t(x) \rvert
&= \lvert \nabla u^\varepsilon_t(x) \rvert + \lvert \nabla U(x) \rvert
\leqslant C_4 (1 + \lvert x\rvert)\,.
\end{align*}
Thus,
\[
\lvert K \star m^\varepsilon_t(x) \cdot \nabla \ln m^\varepsilon_t(x)\rvert
\leqslant C_6 \lvert M\rvert\,.
\]
So the functions $\phi_t$ satisfies
\[
\lVert \phi_t\rVert_{L^\infty}
\leqslant \frac{C_7\lvert M\rvert}{\sqrt{t \wedge 1}}\,.
\]
Therefore, using the convex duality for relative entropy, we get
\[
\frac{\dd \mathcal H^\varepsilon_t}{\dd t}
= - \mathcal I^\varepsilon_t
+ \delta_t \mathcal H^\varepsilon_t
+ \delta_t \ln \int_{\R^{dN}} \exp \biggl( \frac 1{\delta_t (N-1)}
\sum_{i,j=1}^N \phi_t(x^i, x^j) \biggr) (m^\varepsilon_t)^{\otimes N}
(\dd \bm x)\,,
\]
where we set
\[
\delta_t
= \frac{3(1600^2 + 36e^4)C_7 \lvert M\rvert}{\sqrt{t\wedge 1}}\,.
\]
Then, applying the ``concentration'' estimate
\cite[Theorem 4]{JabinWang}
(whose constant is given explicitly in \cite[Theorem 5]{GLBMVortex}), we obtain
\[
\frac{\dd \mathcal H^\varepsilon_t}{\dd t}
\leqslant
- \mathcal I^\varepsilon_t
+ \frac{C_8 \lvert M\rvert}{\sqrt{t\wedge 1}} \mathcal H^\varepsilon_t
+ \frac{C_8 \lvert M\rvert}{\sqrt{t\wedge 1}}
\leqslant
- C_9 \mathcal H^\varepsilon_t
+ \frac{C_8 \lvert M\rvert}{\sqrt{t\wedge 1}} \mathcal H^\varepsilon_t
+ \frac{C_8 \lvert M\rvert}{\sqrt{t\wedge 1}}\,.
\]
We conclude by applying Grönwall's lemma and taking the limit
$\varepsilon \to 0$.
\end{proof}

\begin{proof}[Proof of Theorem~\ref{thm:log-Riesz-quadratic-unif-POC}]
The argument is largely the same as the proof above,
i.e.\ the proof of Theorem~\ref{thm:log-Riesz-global-POC}.
So here we only indicate the differences.
Defining the same $\phi_t$ function as in \eqref{eq:def-phi},
we find that in the quadratic case, we have the following bounds
by Theorem~\ref{thm:MkV-log-Riesz-quadratic-convergence}:
\begin{align*}
\biggl\lVert \nabla \ln \frac{m^\varepsilon_t}{m_*} \biggr\rVert_{L^\infty}
&\leqslant C_1 e^{-ct}\,, \\
\biggl\lVert \nabla^2 \ln \frac{m^\varepsilon_t}{m_*} \biggr\rVert_{L^\infty}
&\leqslant \frac{C_1 e^{-ct}}{\sqrt{t \wedge 1}}\,, \\
\sup_x \lvert x \cdot K^\varepsilon \star (m^\varepsilon_t - m_*)(x)\rvert
&\leqslant C_1 e^{-ct}\,.
\end{align*}
So for the first term in the definition \eqref{eq:def-phi} of $\phi$,
we have, for all $x \neq y$,
\begin{align*}
\MoveEqLeft \bigl|K^\varepsilon(x - y)
\cdot \bigl( \nabla \ln m^\varepsilon_t(x)
- \nabla \ln m^\varepsilon_t(y) \bigr)\bigr| \\
&= \biggl|K^\varepsilon(x - y)
\cdot \biggl( \nabla \ln \frac{m^\varepsilon_t(x)}{m_*(x)}
- \nabla \ln \frac{m^\varepsilon_t(y)}{m_*(y)} \biggr)\biggr| \\
&\lesssim \biggl\lVert \nabla^2 \ln \frac{m^\varepsilon_t}{m_*}
\biggr\rVert_{L^\infty} \leqslant \frac{C_1 e^{-ct}}{\sqrt{t \wedge 1}}\,.
\end{align*}
For the second term, we have
\[
K^\varepsilon \star m^\varepsilon_t (x) \cdot \nabla \ln m^\varepsilon_t(x)
= K^\varepsilon \star (m^\varepsilon - m_*)(x)
\cdot \nabla \ln m_*
+ K^\varepsilon \star m^\varepsilon_t(x)
\cdot \nabla \ln \frac{m^\varepsilon_t(x)}{m_*(x)}\,,
\]
therefore,
\begin{multline*}
\lVert K^\varepsilon \star m^\varepsilon_t \cdot \nabla \ln m^\varepsilon_t
\rVert_{L^\infty}
\leqslant \kappa_U
\sup_x \lvert x \cdot K^\varepsilon \star (m^\varepsilon_t - m_*)(x)\rvert
+ \lVert K^\varepsilon \star m^\varepsilon_t \rVert_{L^\infty}
\biggl\lVert \nabla \ln \frac{m^\varepsilon_t}{m_*} \biggr\rVert_{L^\infty} \\
\leqslant C_2 e^{-ct}\,.
\end{multline*}
Combining the two results above, we derive the decaying $L^\infty$ bound
for $\phi_t$:
\[
\lVert \phi_t \rVert_{L^\infty}
\leqslant \frac{C_3 e^{-ct}}{\sqrt{t \wedge 1}}\,.
\]
Thus, taking the alternative
\[
\delta_t = \frac{3(1600^2+36e^4)C_3 e^{-ct}}{\sqrt{t \wedge 1}}\,,
\]
we get
\[
\frac{\dd \mathcal H^\varepsilon_t}{\dd t}
\leqslant
- \mathcal I^\varepsilon_t
+ \frac{C_4 e^{-ct}}{\sqrt{t\wedge 1}} \mathcal H^\varepsilon_t
+ \frac{C_4 e^{-ct}}{\sqrt{t\wedge 1}}\,.
\]
Finally, we note that, as the Lipschitz constant of
$\ln m^\varepsilon_t / m_*$ tends to zero exponentially,
the perturbed measure $m^\varepsilon_t$ satisfies a $k_t$-LSI with
\[
k_t = 2 \kappa_U \exp ( - C_5 e^{-c't})\,.
\]
Thus, for all $t \geqslant 0$, we have
\[
\mathcal I^\varepsilon_t
\geqslant 2 \kappa_U \exp ( - C_5 e^{-c't})
\mathcal H^\varepsilon_t\,.
\]
We conclude by applying Grönwall's lemma
and taking the limit $\varepsilon \to 0$.
\end{proof}

\appendix

\section{Well-posedness of singular dynamics}
\label{app:log-Riesz-wp}

The mean field well-posedness proof will mainly be based on
the estimates on the convolution with the kernel $K$
in Proposition~\ref{prop:log-Riesz-control-L1-Lp}
and the following elementary result.

\begin{prop}[Growth and stability estimates]
\label{prop:growth-stability}
Let $T > 0$ and $\beta : [0,T] \times \R^d \to \R^d$ be a vector field
that is the sum of a Lipschitz and a bounded part, that is,
$\beta = \beta_\textnormal{Lip} + \beta_\textnormal{b}$
with $\nabla \beta_\textnormal{Lip}$, $\beta_\textnormal{b} \in L^\infty$.
Suppose its divergence is lower bounded:
$(\nabla \cdot \beta)_- \in L^\infty$.
Let $m : [0,T] \to \mathcal P(\mathbb R^d)$ be a probability solution
to the parabolic equation
\[
\partial_t m_t = \Delta m_t - \nabla \cdot (\beta_t m_t)\,.
\]
Then, for all $p \in [2,\infty]$, we have
\[
\lVert m_t \rVert_{L^p} \leqslant C_p \bigl(\lVert m_0 \rVert_{L^p}+1\bigr)
\]
for some $C_p$ depending only on $p$, $d$
and $\lVert (\nabla \cdot \beta)_- \rVert_{L^\infty}$
(notably independent of $t$ and $T$).

Moreover, let $\beta'$ be another vector field satisfying the same conditions
as $\beta$, and let $m'$ be a probability solution to the equation
corresponding to $\beta'$.
Then, for all $p \in \{1\} \cup [2,\infty)$, we have
\[
\lVert m_t - m'_t \rVert_{L^p}
\leqslant e^{C'_p t} \lVert m_0 - m'_0 \rVert_{L^p}
+ C'_p \Bigl( (e^{C'_p t} - 1)\1_{p\geqslant 2} + \sqrt t\1_{p=1}\Bigr)
\sup_{v \in [0,t]} \lVert \beta_v - \beta'_v \rVert_{L^\infty}
\]
for some $C'_p$ depending only on $p$, $d$,
$\lVert (\nabla \cdot \beta)_- \rVert_{L^\infty}$,
$\lVert (\nabla \cdot \beta')_- \rVert_{L^\infty}$,
$\lVert m_0 \rVert_{L^p}$ and $\lVert m'_0 \rVert_{L^p}$.
\end{prop}

\begin{proof}
First, consider the SDE
\[
\dd X_t = \beta_t (X_t) \dd t + \sqrt 2 \dd B_t\,.
\]
Since its drift is the sum of a bounded and a Lipschitz part,
we have the existence of the strong solution and we find that
if $\Law(X_0) = m_0$, then we have the correspondence $\Law(X_t) = m_t$,
by the uniqueness of the PDE.
Moreover, it is known (see e.g.\ \cite{ChampagnatJabinRough}) that if we take
a mollified sequence approaching towards $\beta$,
the SDE solution will also tend to the original one, i.e.\ $X$,
and we have the continuous dependency on the initial value as well.
So without loss of generality,
we can suppose that $\nabla \beta \in \mathcal C^\infty_\textnormal{b}$
and $m_0$ belongs to the Schwartz class.
By a Feynman--Kac argument similar to that
of Proposition~\ref{prop:feynman-kac},
we know that $m_t$ belongs also to the Schwartz class.
Thus, in the following we perform only formal calculations.

\proofstep{Step 1: Growth estimates}
Let $p \geqslant 2$. The $L^p$ norm of $m_t$ satisfies
\begin{align*}
\frac{\dd}{\dd t}\int_{\R^d} m_t^p
& = p \int_{\R^d}
m_t^{p-1} \partial_t m_t \\
& = p \int_{\R^d} m_t^{p-1}
\Bigl( \Delta m_t
- \nabla \cdot ( \beta_t m_t ) \Bigr) \\
& = \int_{\R^d}
\Bigl( - p(p-1)m_t^{p-2} \lvert \nabla m_t \rvert^2
- (p-1)(\nabla \cdot \beta_t) m_t \Bigr) \\
& \leqslant
- p(p-1) \int_{\R^d}
m_t^{p-2} \lvert \nabla m_t \rvert^2
+ (p-1) \lVert (\nabla \cdot \beta)_- \rVert_{L^\infty}\int_{\R^d} m_t^p
\\
& \leqslant
(p-1) \lVert (\nabla \cdot \beta)_- \rVert_{L^\infty} \int_{\R^d} m_t^p
\,,
\end{align*}
where here and in the following $C_p$ denotes a constant
having the same dependencies as in the statement,
and may change from line to line.
We would also denote by $C$ a constant that does not depend on $p$,
but having the same other dependencies.
By Grönwall's lemma, we get
\begin{equation}
\label{eq:MkV-log-Riesz-Lp-growth-short}
\lVert m_t \rVert_{L^p}
\leqslant \exp \biggl( \frac{p-1}{p} (\nabla \cdot \beta)_- \rVert_{L^\infty}
t \biggr) \lVert m_0 \rVert_{L^p}\,,
\end{equation}
and taking $p \to \infty$, we get
\begin{equation}
\label{eq:MkV-log-Riesz-Linfty-growth-short}
\lVert m_t \rVert_{L^\infty}
\leqslant \exp \bigl( \lVert (\nabla \cdot \beta)_- \rVert_{L^\infty} t \bigr)
\lVert m_0 \rVert_{L^\infty}\,.
\end{equation}

Now we show that the two estimates above can be improved
into time-uniform ones.
To this end, define the operator
$\mathcal L_t = \Delta + \beta_t \cdot \nabla$
and its dual
$\mathcal L^*_t = \Delta
- \nabla \cdot ( \beta_t \cdot )$.
Denote by $(P_{u,t})_{0 \leqslant u \leqslant t \leqslant T}$
the time-dependent semi-group generated.
Specializing to $p = 2$ in the $L^p$ computations above, we get
\[
\frac{\dd}{\dd t} \int_{\mathbb R^d} m_t^2
\leqslant - 2 \int_{\mathbb R^d} \lvert \nabla m_t \rvert^2
+ \lVert (\nabla \cdot \beta)_- \rVert_{L^\infty} \int_{\mathbb R^d} m_t^2\,.
\]
The Nash inequality indicates
\[
\lVert m_t \rVert_{L^2}^{1 + 2/d}
\leqslant C_d \lVert m_t \rVert_{L^1}^{2/d} \lVert \nabla m_t \rVert_{L^2}\,,
\]
where $C_d$ depends only on $d$.
So by Grönwall's lemma,
we get the uniform-in-time bound over $\lVert m_t \rVert_{L^2}$:
\[
\lVert m_t \rVert_{L^2}^2
\leqslant \biggl( \frac{C_d^2 \lVert (\nabla\cdot\beta)_- \rVert_{L^\infty}}{2}
\biggr)^{\!d/2}
( 1 - e^{- \kappa t} )^{-d/2}
\lVert m_t \rVert_{L^1}^2\,,
\]
for $\kappa = 2 \lVert (\nabla\cdot\beta)_- \rVert_{L^\infty} / d$.
Note that this bound is independent of
$\lVert m_0 \rVert_{L^\infty}$.
Now we take an arbitrary $h_0 : \mathbb R^d \to [0, \infty)$
of the Schwartz class
and consider the dual evolution $\partial_u h_u = \mathcal L_{t-u} h_u$,
that is,
\[
\partial_u h_u = \Delta h_u + \beta_{t-u} \cdot \nabla h_u\,,
\]
for $u \in [0,t]$, where $t \in [0,T]$.
Deriving the $L^1$ norm of $h_s$ and integrating by parts, we get
\[
\lVert h_u \rVert_{L^1}
\leqslant \exp \bigl( \lVert (\nabla \cdot \beta)_- \rVert_{L^\infty} u \bigr)
\lVert h_0 \rVert_{L^1}\,.
\]
Doing the same for the $L^2$ norm, we get
\begin{align*}
\frac{\dd}{\dd u} \int_{\mathbb R^d} h_u^2
&= 2 \int_{\mathbb R^d} h_u \mathcal L_{T-u} h_u \\
&= 2 \int_{\mathbb R^d} h_u \bigl( \Delta h_u
+ \beta_{t-u} \cdot \nabla h_u \bigr) \\
&= - 2 \int_{\mathbb R^d} \lvert \nabla h_u \rvert^2
- \int_{\mathbb R^d} h_u^2 \nabla \cdot \beta_{t-u} \\
&\leqslant - 2 \int_{\mathbb R^d} \lvert \nabla h_u \rvert^2
+ \lVert (\nabla\cdot\beta)_- \rVert_{L^\infty} \int_{\mathbb R^d} h_u^2\,.
\end{align*}
Again, using the Nash inequality:
\[
\lVert h_u \rVert_{L^2}^{1+2/d}
\leqslant C_d \lVert h_u \rVert_{L^1}^{2/d} \lVert \nabla h_u \rVert_{L^2}
\leqslant
C_d
\exp \bigl( 2 \lVert (\nabla\cdot\beta)_- \rVert_{L^\infty} u / d \bigr)
\lVert h_0 \rVert_{L^1}^{2/d}
\lVert \nabla h_u \rVert_{L^2}\,,
\]
we derive the bound over $\lVert h_t\rVert_{L^2}$:
\[
\lVert P_{t-u, t} h_0 \rVert_{L^2}^2
= \lVert h_u \rVert_{L^2}^2
\leqslant \biggl( \frac{C_d^2 \lVert (\nabla\cdot\beta)_-\rVert_{L^\infty}}{2}
\biggr)^{\!d/2}
( e^{-\kappa u} - e^{-2\kappa u} )^{-d/2}
\lVert h_0 \rVert_{L^1}^2\,,
\]
from which follows the bound on $\lVert P_{t-u, t} \rVert_{L^1\to L^2}$.
So, taking $u = \max ( t/2, t - \kappa^{-1} ) $, we get
\begin{multline}
\label{eq:MkV-log-Riesz-Linfty-growth-long}
\lVert m_t \rVert_{L^\infty}
= \lVert P^*_{u,t} m_{u}\rVert_{L^\infty}
\leqslant \lVert m_u \rVert_{L^2}
\lVert P^*_{u,t} \rVert_{L^2 \to L^\infty} \\
= \lVert m_u \rVert_{L^2} \lVert P_{u,t} \rVert_{L^1 \to L^2}
\leqslant C
(t \wedge 1)^{-d/2} \lVert m_0 \rVert_{L^1}\,.
\end{multline}
So, combining \eqref{eq:MkV-log-Riesz-Linfty-growth-short}
and \eqref{eq:MkV-log-Riesz-Linfty-growth-long}\,,
we get a uniform-in-time bound over $\lVert m_t \rVert_{L^\infty}$:
\begin{equation}
\label{eq:MkV-log-Riesz-Linfty-growth}
\sup_{t \in [0,T]}
\lVert m_t \rVert_{L^\infty}
\leqslant C
\bigl( \lVert m_0 \rVert_{L^1} + \lVert m_0 \rVert_{L^\infty} \bigr)\,.
\end{equation}
Finally, by differentiating $\int m_t$ and integrating by parts,
we get
\begin{equation}
\label{eq:MkV-log-Riesz-L1-growth}
\lVert m_t \rVert_{L^1}
= \lVert m_0 \rVert_{L^1}\,.
\end{equation}
Similarly, interpolating between \eqref{eq:MkV-log-Riesz-Linfty-growth-long}
and \eqref{eq:MkV-log-Riesz-L1-growth}, we get
\[
\lVert m_t \rVert_{L^p}
\leqslant C^{(p-1)/p} (t \wedge 1)^{-(p-1)d/2p}
\lVert m_0 \rVert_{L^1}\,,
\]
and combing with \eqref{eq:MkV-log-Riesz-Lp-growth-short}, we get
\begin{equation}
\label{eq:MkV-log-Riesz-Lp-growth}
\sup_{t \in [0,T]}
\lVert m_t \rVert_{L^p}
\leqslant C_p
\bigl( \lVert m_0 \rVert_{L^1} + \lVert m_0 \rVert_{L^p} \bigr)\,.
\end{equation}

\proofstep{Step 2: Stability estimates}
Now let $\beta'$, $m'$ be the other vector field and the probability solution.
Recall that $m$, $m'$ correspond respectively to the SDE
\begin{alignat*}{2}
\dd X_t &= \beta_t(X_t) \dd t
+ \sqrt 2 \dd W_t\,,&\qquad\Law(X_0) &= m_0\,, \\
\dd X'_t &= \beta'_t(X'_t) \dd t
+ \sqrt 2 \dd W_t\,,&\Law(X'_0) &= m'_0\,. \\
\intertext{%
Now introduce the third SDE, whose drift term is identical to the first,
but initial condition identical to the second:}
\dd X''_t &= \beta_t(X''_t) \dd t
+ \sqrt 2 \dd W_t\,,&\Law(X''_0) &= m'_0\,.
\end{alignat*}
such that $\Proba[X_0 \neq X''_0] = \frac 12 \lVert m_0 - m'_0 \rVert_{L^1}$.
Denote $m''_t = \Law(X''_t)$.
Thus, conditioning on the initial condition, we get
\begin{align*}
\lVert m_t - m''_t \rVert_{L^1}
&\leqslant 2\Proba[X_t \neq X''_t]
\leqslant 2\Proba[X_0 \neq X''_0]
= \lVert m_0 - m'_0 \rVert_{L^1}\,.
\intertext{%
On the other hand, by Pinsker's inequality and Girsanov's theorem, we have}
\lVert m'_t - m''_t \rVert_{L^1}^2
&\leqslant 2 H( m'_t | m''_t )
\leqslant \frac 12 \int_0^t \lVert \beta_v-\beta'_v\rVert_{L^\infty}^2 \dd v\,.
\end{align*}
Combining the two inequalities above yields the $L^1$-stability estimate.
Now, let $p \geqslant 2$ and let us calculate:
\begin{align*}
\MoveEqLeft \frac{\dd}{\dd t} \int_{\mathbb R^d} \lvert m_t - m'_t \rvert^p \\
&= p \int_{\mathbb R^d} \lvert m_t - m'_t \rvert^{p-2}
(m_t - m'_t)
\Bigl( \Delta(m_t - m'_t)
- \nabla \cdot (m_t \beta_t)
+ \nabla \cdot (m'_t \beta'_t) \Bigr) \\
&=
- p(p-1) \int_{\mathbb R^d} \lvert m_t - m'_t \rvert^{p-2}
\lvert \nabla (m_t - m'_t) \rvert^2 \\
&\hphantom{={}}\quad
+ p \int_{\mathbb R^d} \lvert m_t - m'_t \rvert^{p-2}
(m_t - m'_t)
\Bigl( - \nabla \cdot (m_t \beta_t)
+ \nabla \cdot (m'_t \beta'_t) \Bigr) \\
&=
- p(p-1) \int_{\mathbb R^d} \lvert m_t - m'_t \rvert^{p-2}
\lvert \nabla (m_t - m'_t) \rvert^2 \\
&\hphantom{={}}\quad
+ p \int_{\mathbb R^d} \lvert m_t - m'_t \rvert^{p-2}
(m_t - m'_t) \Bigl( - \nabla \cdot \bigl(
(m_t - m'_t) \beta'_t\bigr)
+ \nabla \cdot \bigl(m_t (\beta'_t - \beta_t)\bigr) \Bigr) \\
&\leqslant
(p-1) \int_{\mathbb R^d} \lvert m_t - m'_t \rvert^p
\cdot (- \nabla \cdot \beta'_t)
+ \frac {p(p-1)}{4}
\int_{\mathbb R^d} \lvert m_t - m'_t \rvert^{p-2}
m_t^2 \lvert \beta_t - \beta'_t \rvert^2 \\
&\leqslant
\frac{(p-1)}{2} \lVert (\nabla\cdot\beta')_- \rVert_{L^\infty}
\int_{\mathbb R^d} \lvert m_t - m'_t \rvert^p
+ \frac{(p-1)(p-2)}{4} \int_{\mathbb R^d} \lvert m_t - m'_t \rvert^p \\
&\hphantom{\leqslant{}}\quad
+ \frac{p-1}{2} \lVert m_t \rVert_{L^p}^p
\lVert \beta_t - \beta'_t \rVert_{L^\infty}^p\,.
\end{align*}
Then, using the uniform $L^p$ estimate in the first step,
applying Grönwall's lemma and taking the $p$-th root, we get
the desired result.
\end{proof}

Now we are ready to prove the well-posedness of the mean field dynamics.

\begin{proof}[Proof of Proposition~\ref{prop:MkV-log-Riesz-wp-reg}]

Take a $p \in (1,\infty)$ such that $p^{-1} < 1 - \frac{s+1}{d}$
and let $q$ be its conjugate: $p^{-1} + q^{-1} = 1$.
We also take a $\theta \in (0, d - s - 1)$.

\proofstep{Step 1: Well-posedness}
Let $T \in (0,\infty)$.
We define the functional space:
\[
\mathcal X \coloneqq \mathcal C( [0,T]; L^1\cap L^p \cap \mathcal P)\,.
\]
The space $\mathcal X$ is a complete metric space.
Given $m \in X$, we let $\mathcal T[m]$ be the uniqueness probability
solution to the Cauchy problem
\[
\partial_t \mathcal T[m]_t = \Delta \mathcal T[m]_t
- \nabla \cdot \bigl( (K \star m_t - \nabla U) \mathcal T[m]_t \bigr)\,,\quad
\mathcal T[m]_0 = m_0\,.
\]
According to Proposition~\ref{prop:growth-stability} we know that
$\mathcal T[m] \in \mathcal X$, where the continuity
in $L^1 \cap L^p$ follows from a density argument.
Moreover, by the stability estimate in the proposition,
for all $m$, $m' \in X$, we have
\begin{multline*}
\lVert \mathcal T[m]_t - \mathcal T[m']_t \rVert_{L^r}
\leqslant e^{C'_r t} \lVert m_0 - m'_0 \rVert_{L^r} \\
+ C'_r \Bigl( (e^{C'_r t} - 1)\1_{r=p} + \sqrt t\1_{r=1}\Bigr)
\sup_{v \in [0,t]} \lVert K \star (m_v - m'_v) \rVert_{L^\infty}
\end{multline*}
for $r = 1$, $p$.
But by Proposition~\ref{prop:log-Riesz-control-L1-Lp}, we have
\[
\lVert K \star (m_v - m'_v) \rVert_{L^\infty}
\lesssim \lVert m_v - m'_v \rVert_{L^1}^{1-q(s+1)/d}
\lVert m_v - m'_v \rVert_{L^p}^{q(s+1)/d}\,.
\]
Thus, restricting to the subspace of $\mathcal X$ of common initial value
and letting $T$ be small enough,
we get that the mapping $\mathcal T$ is a contraction in $\mathcal X$.
So a time-local solution exists and is unique.
Thanks to the uniform growth estimates,
this short time interval can be extended infinitely by iteration.
So a unique global solution is recovered
and it satisfies the uniform $L^\infty$ bound
thanks to Proposition~\ref{prop:growth-stability}.
For the continuous dependency on the initial value,
we use the stability estimates on a small time interval
without restricting the initial values to be the same
and iterate infinitely as well.

\proofstep{Step 2: Control of moments}
Given the uniform $L^\infty$ bound obtained above,
we have, according to Proposition~\ref{prop:log-Riesz-control-L1-Lp},
\[
\lVert K \star m_t \rVert_{L^\infty}
\lesssim \lVert m_t \rVert_{L^1}^{1 - (s+1)/d}
\lVert m_t \rVert_{L^\infty}^{(s+1)/d}.
\]
So the contribution from the interaction kernel is bounded.
Then we construct, for $k > 0$, the Lyapunov function
\[
V_k(x) = \sqrt{1 + \lvert x \rvert^{2k}}\,,
\]
and we can easily verify
\[
\bigl( \Delta - \nabla U \cdot \nabla
+ (K \star m_t) \cdot \nabla \bigr) V_k
\leqslant - c_k V_k + C_k\,,
\]
for some $c_k > 0$, $C_k \geqslant 0$.
This implies the uniform bound on the $k$-th moment.

\proofstep{Step 3: Approximation}
Let $(m^\varepsilon_t)_{t \geqslant 0}$ be the flow
corresponding to the mollified kernel $K^\varepsilon$
and potential $U^\varepsilon$.
Applying the stability estimates in Proposition~\ref{prop:growth-stability},
we get
\begin{multline*}
\lVert m_t - m^\varepsilon_t \rVert_{L^r}
\leqslant e^{C'_r t} \lVert m_0 - m^\varepsilon_0 \rVert_{L^r} \\
+ C'_r \Bigl( (e^{C'_r t} - 1)\1_{r=p} + \sqrt t\1_{r=1}\Bigr)
\sup_{v \in [0,t]}
\Bigl(\lVert K \star m_v - K^\varepsilon\star m^\varepsilon_v\lVert_{L^\infty}
+ \lVert\nabla U - \nabla U^\varepsilon\rVert_{L^\infty}\Bigr)\,.
\end{multline*}
Note that the initial $L^p$ error
$\lVert m_t - m^\varepsilon_t\rVert_{L^p} \to 0$
by interpolation between $L^1$ and $L^\infty$.
For the first term in the supremem,
we have
\begin{align*}
\lVert K \star m_v - K^\varepsilon\star m^\varepsilon_v\lVert_{L^\infty}
&\leqslant \lVert K \star (m_v - m^\varepsilon_v) \rVert_{L^\infty}
+ \lVert K \star m^\varepsilon_v -
K \star m^\varepsilon_v\star \eta^\varepsilon \rVert_{L^\infty} \\
&\leqslant \lVert K \star (m_v - m^\varepsilon_v) \rVert_{L^\infty}
+ \varepsilon^\theta [K\star m^\varepsilon_v]_{\mathcal C^\theta}\,.
\end{align*}
By the $L^\infty$ and Hölder estimates
in Proposition~\ref{prop:log-Riesz-control-L1-Lp},
we have the following controls:
\begin{align*}
\lVert K \star (m_v - m^\varepsilon_v) \rVert_{L^\infty}
&\lesssim \lVert m_v - m^\varepsilon_v \rVert_{L^1}^{1-q(s+1)/d}
\lVert m_v - m^\varepsilon_v \rVert_{L^p}^{q(s+1)/d}\,, \\
[K\star m^\varepsilon_v]_{\mathcal C^\theta}
&\lesssim \lVert m^\varepsilon_v\rVert_{L^1}^{1-(s+1+\theta)/d}
\lVert m^\varepsilon_v\rVert_{L^\infty}^{(s+1+\theta)/d}\,.
\end{align*}
For the second term we simply bound
$\lVert \nabla U - \nabla U^\varepsilon\rVert_{L^\infty}
\leqslant \lVert \nabla^2U\rVert_{L^\infty}\varepsilon$.
Since $m^\varepsilon$ is again uniformly bounded in $L^1\cap L^\infty$,
we get an error bound between $m_t$ and $m^\varepsilon_t$ for small $t$
and we iterate infinitely.
\end{proof}

Finally, we prove the well-posedness of the particle system
in the non-attractive sub-Coulombic and Coulombic cases.

\begin{proof}[Proof of Proposition~\ref{prop:ps-log-Riesz-wp}]
Define for $n \in \mathbb N$ the sequence of stopping times:
\[
\tau_n \coloneqq \inf \bigl\{ t \geqslant 0 :
\bigl| X^i_t - X^j_t \bigr| \leqslant 1/n
~\text{for some}~i \neq j \bigr\}\,.
\]
Then the original SDE system \eqref{eq:ps-log-Riesz-SDE} stopped at $\tau_n$
is well defined according to Cauchy--Lipschitz theory.
Consider the ``energy'' functional
\[
E (\bm x)
= E(x^1, \ldots, x^N)
= \frac 12 \sum_{\substack{i,j \in \cco 1, N\ccf\\ i \neq j}} g_s(x^i - x^j)
+ \frac {N\1_{s = 0}}2 \sum_{i=1}^N\,\lvert x^i \rvert^2\,.
\]
The energy functional is always lower bounded, and by It\=o calculus,
we find that $\Expect \bigl[ E\bigl(\bm X_{t \wedge \tau_n}\bigr)\bigr]$
is upper bounded uniformly in $n$.
Then using the Markov inequality for the energy,
we show that $\Proba [ \tau_n \leqslant t ] \to 0$ when $n \to \infty$.
This implies that $\lim_{n \to \infty} \tau_n = \infty$ almost surely,
thus the local well-posedness of the SDE extends to the half line $[0,\infty)$.
That is to say the first claim is proved.

Now prove the second claim.
For each $n \in \mathbb N$, we construct a Lipschitz kernel
$\tilde K_n : \mathbb R^d \to \mathbb R$
such that $\tilde K_n(x) = K(x)$ for $x \in \mathbb R^d$
with $\lvert x \rvert \geqslant 1/n$.
Define the convolution
$\tilde K_n^\varepsilon = \tilde K_n \star \eta^\varepsilon$
and consider the SDE system
\[
\dd \tilde X^{\varepsilon,i}_{n,t}
= -\nabla U^\varepsilon \bigl( \tilde X^{\varepsilon,i}_{n,t} \bigr) \dd t
+ \frac 1{N-1} \sum_{j \in \cco 1, N\ccf \setminus\{i\}}
\tilde K^\varepsilon_n
\bigl(\tilde X^{\varepsilon,i}_{n,t}
- \tilde X^{\varepsilon,j}_{n,t}\bigr) \dd t
+ \sqrt 2 \dd W^i_t\,,
\]
for $i \in \cco 1, N \ccf$,
with initial condition $\tilde{\bm X}^\varepsilon_{n,0} = \bm X_0$.
Define the stopping time
\[
\tau_n^\varepsilon \coloneqq \inf \bigl\{ t \geqslant 0 :
\bigl| \tilde{X}^{\varepsilon,i}_{n,t}
- \tilde{X}^{\varepsilon,j}_{n,t} \bigr|
\leqslant 1/n + \varepsilon ~\text{for some}~i \neq j \bigr\}\,.
\]
By construction, we know
\[
\tilde {\bm X}^\varepsilon_{n,t \wedge \tau_{n}^\varepsilon}
= {\bm X}^\varepsilon_{t \wedge \tau_{n}^\varepsilon}~\text{a.s.}
\]
On the other hand, by Cauchy--Lipschitz theory, we know
\[
\sup_{t \in [0,T]}
\bigl| \tilde{\bm X}^{\varepsilon}_{n,t\wedge \tau_n}
- \bm{X}_{t \wedge \tau_n}\!\bigr|
\leqslant C(n, N, K, U, T) \varepsilon~\text{a.s.}
\]
Thus, for each $n \in \mathbb N$,
there exists $\varepsilon_0 (n, N, K, U, T) > 0$ such that
for all $\varepsilon \leqslant \varepsilon_0$, we have
\[
\sup_{t \in [0,T]}
\bigl| \tilde{\bm X}^{\varepsilon}_{n,t\wedge \tau_n}
- \bm{X}_{t \wedge \tau_n}\!\bigr|
\leqslant \frac 1{3n}~\text{a.s.}
\]
In particular, we get
for all $\varepsilon \leqslant \varepsilon_0$,
$t \leqslant T\wedge \tau_n$ and $i \neq j$,
\[
\bigl| \tilde{X}^{\varepsilon,i}_{n,t}
- \tilde{X}^{\varepsilon,j}_{n,t} \bigr|
\geqslant \frac 1{3n}~\text{a.s.}
\]
Consequently, for $\varepsilon
\leqslant \varepsilon_1 (n, N, K, U, T)
\coloneqq \varepsilon_0 \wedge 1/(13n)$, we have
$T \wedge \tau_n \leqslant \tau_{4n}^\varepsilon$, and therefore,
\begin{multline*}
\sup_{t \leqslant T \wedge \tau_n}
\bigl| \bm X^{\varepsilon}_{t}
- \bm X_t \bigr|
= \sup_{t \leqslant T \wedge \tau_n}
\bigl| \bm X^{\varepsilon}_{t \wedge \tau^{\varepsilon}_{4n}}
- \bm X_t \bigr|
= \sup_{t \leqslant T \wedge \tau_n}
\bigl| \tilde{\bm X}^{\varepsilon}_{4n,t}
- \bm X_t \bigr| \\
\leqslant C(4n, N, K, U, T) \varepsilon~\text{a.s.}
\end{multline*}
Thus, taking $\varepsilon \to 0$, we get
$\bm X^\varepsilon_{t\wedge\tau_n}
\to \bm X_{t\wedge\tau_n}$ a.s.\ for all $t \leqslant T$.
We recover the second claim by using the arbitrariness of $T$
and the fact that $\lim_{n \to \infty} \tau_n = \infty$ a.s.
\end{proof}

\section{Feynman--Kac formula}
\label{app:feynman-kac}

\begin{prop}
\label{prop:feynman-kac}
Let $T > 0$.
Suppose $\beta : [0,T] \times \mathbb R^d \to \mathbb R^d$
and $\varphi : [0,T] \times \mathbb R^d \to \mathbb R$
are measurable functions and suppose that there exists $C > 0$ such that
for all $t \in [0,T]$ and $x \in \mathbb R^d$, we have
\begin{align*}
\lvert \beta(t,x) \rvert &\leqslant C (1 + \lvert x\rvert)\,, \\
\lvert \varphi(t,x) \rvert &\leqslant C (1 + \lvert x\rvert)\,, \\
\bigl| \nabla^k_x \beta(t,x) \bigr| &\leqslant C\,,
\quad\text{for}\ k \in \cco 1, 3\ccf\,, \\
\bigl| \nabla^k_x \varphi(t,x) \bigr| &\leqslant C\,,
\quad\text{for}\ k \in \cco 1, 2\ccf\,.
\end{align*}
Suppose in addition that $f_0 : \mathbb R^d \to \mathbb R$ is measurable
and satisfies, for the same constant $C$,
and for all $x \in \mathbb R^d$,
\[
\lvert \nabla^k f_0(x) \rvert \leqslant
C \exp \bigl(C (1 + \lvert x\rvert)\bigr)\,,
\quad\text{for}\ k \in \cco 0, 2\ccf\,.
\]
Then, the function $f : [0,T] \times \mathbb R^d \to \mathbb R$ defined by
\[
f(t,x) = \Expect \biggl[ \exp \biggl(
\int_0^t \varphi\bigl(t-u, X^{t,x}_u\bigr) \dd u \biggr)
f_0 \bigl(X^{t,x}_t\bigr)\biggr]\,,
\]
where $X^{t,x}_\cdot$ solves
\[
\dd X^{t,x}_u = \beta\bigl(t-s, X^{t,x}_u\bigr) \dd u + \sqrt 2 \dd B_t\,,
\quad u \in [0,t]\,,
\qquad X^{t,x}_0 = x\,,
\]
is a strong solution to the Cauchy problem
\[
\partial_t f = \Delta f + \beta \cdot \nabla f + \varphi f\,,
\qquad f |_{t = 0} = f_0
\]
with the following bound: there exists $C' > 0$ such that
for all $t \in [0,T]$ and $x \in \mathbb R^d$, we have
\[
\lvert \nabla^k f_0(t,x) \rvert \leqslant
C' \exp \bigl(C' (1 + \lvert x\rvert)\bigr)\,,
\quad\text{for}\ k \in \cco 0, 2\ccf\,.
\]
\end{prop}

The result can be easily obtained by differentiating
the defining SDE of the process $X^{t,x}_\cdot$.
We refer readers to e.g.\ the appendix of \cite{ulpoc} for details.

\paragraph{Acknowledgments.}
The work of P.\,M. is supported by the Project CONVIVIALITY ANR-23-CE40-0003
of the French National Research Agency.
The work of Z.\,R. is supported by Finance For Energy Market Research Centre.

\printbibliography

@online{lsihe,
	title={Logarithmic {Sobolev} inequalities for non-equilibrium steady states},
	author={Pierre Monmarch\'e and Songbo Wang},
	year={2023},
	eprint={2309.07858},
	archivePrefix={arXiv},
	primaryClass={math.PR}
}

@article{ulpoc,
	title={Uniform-in-time propagation of chaos for mean field {Langevin} dynamics},
	author={Fan Chen and Zhenjie Ren and Songbo Wang},
	year={2024},
	journal={Ann.\ Inst.\ Henri Poincaré, Probab.\ Stat.},
	eprint={2212.03050},
	archivePrefix={arXiv},
	primaryClass={math.PR},
	pubstate={forthcoming}
}

@Article{JabinWang,
 Author = {Jabin, Pierre-Emmanuel and Wang, Zhenfu},
 Title = {Quantitative estimates of propagation of chaos for stochastic systems
          with ${W}^{-1,\infty}$ kernels},
 FJournal = {Inventiones Mathematicae},
 Journal = {Invent. Math.},
 ISSN = {0020-9910},
 Volume = {214},
 Number = {1},
 Pages = {523--591},
 Year = {2018},
 Language = {English},
 DOI = {10.1007/s00222-018-0808-y},
 zbMATH = {6955480},
 Zbl = {1402.35208}
}

@Article{LackerHier,
 Author = {Lacker, Daniel},
 Title = {Hierarchies, entropy, and quantitative propagation of chaos
          for mean field diffusions},
 FJournal = {Probability and Mathematical Physics},
 Journal = {Probab. Math. Phys.},
 ISSN = {2690-0998},
 Volume = {4},
 Number = {2},
 Pages = {377--432},
 Year = {2023},
 Language = {English},
 DOI = {10.2140/pmp.2023.4.377},
 zbMATH = {7709553},
 Zbl = {1515.82109}
}

@Article{SerfatyME,
 Author = {Serfaty, Sylvia},
 Title = {Mean field limit for {Coulomb}-type flows},
 FJournal = {Duke Mathematical Journal},
 Journal = {Duke Math. J.},
 ISSN = {0012-7094},
 Volume = {169},
 Number = {15},
 Pages = {2887--2935},
 Year = {2020},
 Language = {English},
 DOI = {10.1215/00127094-2020-0019},
 zbMATH = {7292321},
 Zbl = {1475.35341},
 Note = {Appendix by Mitia Duerinckx and Sylvia Serfaty}
}

@inproceedings{BJWMFE,
  title={Modulated free energy and mean field limit},
  author={Bresch, Didier and Jabin, Pierre-Emmanuel and Wang, Zhenfu},
  booktitle={S{\'e}min. Laurent Schwartz, EDP Appl., Ann{\'e}e 2019-2020},
  eid={Expos{\'e} n\textsuperscript{o}~II},
  doi={10.5802/slsedp.135}
}

@article{BJWAttractive,
    AUTHOR = {Bresch, Didier and Jabin, Pierre-Emmanuel and Wang, Zhenfu},
     TITLE = {Mean field limit and quantitative estimates with singular
              attractive kernels},
   JOURNAL = {Duke Math. J.},
  FJOURNAL = {Duke Mathematical Journal},
    VOLUME = {172},
      YEAR = {2023},
    NUMBER = {13},
     PAGES = {2591--2641},
      ISSN = {0012-7094,1547-7398},
   MRCLASS = {35Q82 (82C22)},
  MRNUMBER = {4658923},
       DOI = {10.1215/00127094-2022-0088}
}

@article{GLBMVortex,
	title = {Uniform in time propagation of chaos for the {2D} vortex model and other singular stochastic systems},
	journal = {J. Eur. Math. Soc.},
	author = {Guillin, Arnaud and Le Bris, Pierre and Monmarché, Pierre},
	date = {2024-01-22},
    pubstate = {prepublished},
    doi = {10.4171/JEMS/1413}
}

@article{CdCRSUniform,
author={Chodron de Courcel, Antonin and Rosenzweig, Matthew and Serfaty, Sylvia},
title={Sharp uniform-in-time mean-field convergence for singular periodic {Riesz} flows},
journal={Ann. Inst. Henri Poincaré, Anal. Non Linéaire},
date={2023-12-01},
pubstate={prepublished},
doi={10.4171/AIHPC/105}
}

@online{FWVortex,
      title={Quantitative Propagation of Chaos for {2D} Viscous Vortex Model on the Whole Space},
      author={Xuanrui Feng and Zhenfu Wang},
      year={2023},
      eprint={2310.05156},
      archivePrefix={arXiv},
      primaryClass={math.AP}
}

@Article{LLFSharp,
 Author = {Lacker, Daniel and Le Flem, Luc},
 Title = {Sharp uniform-in-time propagation of chaos},
 FJournal = {Probability Theory and Related Fields},
 Journal = {Probab. Theory Relat. Fields},
 ISSN = {0178-8051},
 Volume = {187},
 Number = {1-2},
 Pages = {443--480},
 Year = {2023},
 Language = {English},
 DOI = {10.1007/s00440-023-01192-x},
 zbMATH = {7735855}
}

@Book{BGLMarkov,
 Author = {Bakry, Dominique and Gentil, Ivan and Ledoux, Michel},
 Title = {Analysis and geometry of {Markov} diffusion operators},
 FSeries = {Grundlehren der Mathematischen Wissenschaften},
 Series = {Grundlehren Math. Wiss.},
 ISSN = {0072-7830},
 Volume = {348},
 ISBN = {978-3-319-00226-2; 978-3-319-00227-9},
 Year = {2014},
 Publisher = {Springer},
 Address = {Cham},
 Language = {English},
 DOI = {10.1007/978-3-319-00227-9},
 zbMATH = {6175511},
 Zbl = {1376.60002}
}

@Article{HolleyStroockLSI,
 Author = {Holley, Richard and Stroock, Daniel},
 Title = {Logarithmic {Sobolev} inequalities and stochastic {Ising} models},
 FJournal = {Journal of Statistical Physics},
 Journal = {J. Stat. Phys.},
 ISSN = {0022-4715},
 Volume = {46},
 Number = {5-6},
 Pages = {1159--1194},
 Year = {1987},
 Language = {English},
 DOI = {10.1007/BF01011161},
 zbMATH = {4117639},
 Zbl = {0682.60109}
}

@article{AidaShigekawaLSI,
 Author = {Aida, Shigeki and Shigekawa, Ichiro},
 Title = {Logarithmic {Sobolev} inequalities and spectral gaps: {Perturbation} theory},
 FJournal = {Journal of Functional Analysis},
 Journal = {J. Funct. Anal.},
 ISSN = {0022-1236},
 Volume = {126},
 Number = {2},
 Pages = {448--475},
 Year = {1994},
 Language = {English},
 DOI = {10.1006/jfan.1994.1154},
 zbMATH = {763977},
 Zbl = {0846.46019}
}

@inproceedings{LuTwoScale,
  title = 	 {Two-Scale Gradient Descent Ascent Dynamics Finds Mixed {N}ash Equilibria of Continuous Games: A Mean-Field Perspective},
  author =       {Lu, Yulong},
  booktitle = 	 {Proc. 40th Int. Conf. Mach. Learn.},
  pages = 	 {22790--22811},
  year = 	 {2023},
  volume = 	 {202},
  series = 	 {Proc. Mach. Learn. Res.},
  publisher =    {PMLR},
  url = 	 {https://proceedings.mlr.press/v202/lu23b.html}
}

@article{LRSLongTime,
 Author = {Leli{\`e}vre, Tony and Rousset, Mathias and Stoltz, Gabriel},
 Title = {Long-time convergence of an adaptive biasing force method},
 FJournal = {Nonlinearity},
 Journal = {Nonlinearity},
 ISSN = {0951-7715},
 Volume = {21},
 Number = {6},
 Pages = {1155--1181},
 Year = {2008},
 Language = {English},
 DOI = {10.1088/0951-7715/21/6/001},
 zbMATH = {5291950},
 Zbl = {1146.35320}
}

@online{RenWangEntropy,
      title={Entropy Estimate Between Diffusion Processes with Application to Nonlinear {Fokker}--{Planck} Equations},
      author={Panpan Ren and Feng-Yu Wang},
      year={2023},
      eprint={2302.13500},
      archivePrefix={arXiv},
      primaryClass={math.PR}
}

@Article{MalrieuLSI,
 Author = {Malrieu, Florent},
 Title = {Logarithmic {Sobolev} inequalities for some nonlinear {PDE}'s.},
 FJournal = {Stochastic Processes and their Applications},
 Journal = {Stochastic Processes Appl.},
 ISSN = {0304-4149},
 Volume = {95},
 Number = {1},
 Pages = {109--132},
 Year = {2001},
 Language = {English},
 DOI = {10.1016/S0304-4149(01)00095-3},
 zbMATH = {2138891},
 Zbl = {1059.60084}
}

@ARTICLE{CdRDMTT,
 Author = {Chaudru de Raynal, Paul-{\'E}ric and Duong, Manh Hong and Monmarch{\'e}, Pierre and Toma{\v{s}}evi{\'c}, Milica and Tugaut, Julian},
 Title = {Reducing exit-times of diffusions with repulsive interactions},
 FJournal = {European Series in Applied and Industrial Mathematics (ESAIM): Probability and Statistics},
 Journal = {ESAIM, Probab. Stat.},
 ISSN = {1292-8100},
 Volume = {27},
 Pages = {723--748},
 Year = {2023},
 Language = {English},
 DOI = {10.1051/ps/2023012},
 zbMATH = {7730457},
 Zbl = {1520.60045}
}

@Article{KuwadaDuality,
 Author = {Kuwada, Kazumasa},
 Title = {Duality on gradient estimates and {Wasserstein} controls},
 FJournal = {Journal of Functional Analysis},
 Journal = {J. Funct. Anal.},
 ISSN = {0022-1236},
 Volume = {258},
 Number = {11},
 Pages = {3758--3774},
 Year = {2010},
 Language = {English},
 DOI = {10.1016/j.jfa.2010.01.010},
 zbMATH = {5708624},
 Zbl = {1194.53032}
}

@article{MonmarcheHighTemperature,
     author = {Monmarch{\'e}, Pierre},
     title = {Wasserstein contraction and {Poincar{\'e}} inequalities for elliptic diffusions with high diffusivity},
     journal = {Ann. Henri Lebesgue},
     pages = {941--973},
     publisher = {ENS Rennes},
     volume = {6},
     year = {2023},
     doi = {10.5802/ahl.182},
     language = {English}
}

@article{ConfortiWeakSemiconvexity,
	title = {Weak semiconvexity estimates for {Schr{\"o}dinger} potentials and logarithmic {Sobolev} inequality for {Schr{\"o}dinger} bridges},
	issn = {0178-8051, 1432-2064},
	doi = {10.1007/s00440-024-01264-6},
	language = {English},
	journal = {Probab. Theory Relat. Fields},
	author = {Conforti, Giovanni},
	month = feb,
	year = {2024}
}

@online{DFOParabolicComparison,
      title={Parabolic comparison revisited and applications},
      author={Joscha Diehl and Peter K. Friz and Harald Oberhauser},
      year={2011},
      eprint={1102.5774},
      archivePrefix={arXiv},
      primaryClass={math.AP}
}

@Article{RSGlobalConvergenceRiesz,
 Author = {Rosenzweig, Matthew and Serfaty, Sylvia},
 Title = {Global-in-time mean-field convergence for singular {Riesz}-type diffusive flows},
 FJournal = {The Annals of Applied Probability},
 Journal = {Ann. Appl. Probab.},
 ISSN = {1050-5164},
 Volume = {33},
 Number = {2},
 Pages = {954--998},
 Year = {2023},
 Language = {English},
 DOI = {10.1214/22-AAP1833},
 zbMATH = {7692281},
 Zbl = {1516.35414}
}

@Book{FlemingSoner,
 Author = {Fleming, Wendell H. and Soner, H. Mete},
 Title = {Controlled {Markov} processes and viscosity solutions},
 Edition = {2nd ed.},
 FSeries = {Stochastic Modelling and Applied Probability},
 Series = {Stoch. Model. Appl. Probab.},
 ISSN = {0172-4568},
 Volume = {25},
 ISBN = {0-387-26045-5; 978-1-4419-2078-2; 0-387-31071-1},
 Year = {2006},
 Publisher = {New York, NY: Springer},
 Language = {English},
 DOI = {10.1007/0-387-31071-1},
 zbMATH = {5032360},
 Zbl = {1105.60005}
}

@Article{PorrettaPriolaGlobalLip,
 Author = {Porretta, Alessio and Priola, Enrico},
 Title = {Global {Lipschitz} regularizing effects for linear and nonlinear parabolic equations},
 FJournal = {Journal de Math{\'e}matiques Pures et Appliqu{\'e}es. Neuvi{\`e}me S{\'e}rie},
 Journal = {J. Math. Pures Appl. (9)},
 ISSN = {0021-7824},
 Volume = {100},
 Number = {5},
 Pages = {633--686},
 Year = {2013},
 Language = {English},
 DOI = {10.1016/j.matpur.2013.01.016},
 zbMATH = {6448859},
 Zbl = {1322.35066}
}

@Article{ConfortiCouplReflControlDiffusion,
 Author = {Conforti, Giovanni},
 Title = {Coupling by reflection for controlled diffusion processes: turnpike property and large time behavior of {Hamilton}--{Jacobi}--{Bellman} equations},
 FJournal = {The Annals of Applied Probability},
 Journal = {Ann. Appl. Probab.},
 ISSN = {1050-5164},
 Volume = {33},
 Number = {6A},
 Pages = {4608--4644},
 Year = {2023},
 Language = {English},
 DOI = {10.1214/22-AAP1927},
 zbMATH = {7789643}
}

@article{RosenzweigSerfatyMLSI,
author={Rosenzweig, Matthew and Serfaty, Sylvia},
title={Modulated logarithmic {Sobolev} inequalities and generation of chaos},
journal={Ann. Fac. Sci. Toulouse, Math. (6)},
pubstate={forthcoming},
      year={2023},
      eprint={2307.07587},
      archivePrefix={arXiv},
      primaryClass={math.PR}
}

@unpublished{HRSunpublished,
    author = {Huang, Jiaoyang and  Rosenzweig,  Matthew and Serfaty, Sylvia },
    title ={The modulated free energy method on $\mathbb {R}^d$},
    pubstate={inpreparation}
}

@Article{PriolaWangGradientEstimate,
 Author = {Priola, Enrico and Wang, Feng-Yu},
 Title = {Gradient estimates for diffusion semigroups with singular coefficients},
 FJournal = {Journal of Functional Analysis},
 Journal = {J. Funct. Anal.},
 ISSN = {0022-1236},
 Volume = {236},
 Number = {1},
 Pages = {244--264},
 Year = {2006},
 Language = {English},
 DOI = {10.1016/j.jfa.2005.12.010},
 zbMATH = {5037259},
 Zbl = {1110.47035}
}

@Article{EberleReflectionCoupling,
 Author = {Eberle, Andreas},
 Title = {Reflection couplings and contraction rates for diffusions},
 FJournal = {Probability Theory and Related Fields},
 Journal = {Probab. Theory Relat. Fields},
 ISSN = {0178-8051},
 Volume = {166},
 Number = {3-4},
 Pages = {851--886},
 Year = {2016},
 Language = {English},
 DOI = {10.1007/s00440-015-0673-1},
 zbMATH = {6657178},
 Zbl = {1367.60099}
}

@Article{ChampagnatJabinRough,
 Author = {Champagnat, Nicolas and Jabin, Pierre-Emmanuel},
 Title = {Strong solutions to stochastic differential equations with rough coefficients},
 FJournal = {The Annals of Probability},
 Journal = {Ann. Probab.},
 ISSN = {0091-1798},
 Volume = {46},
 Number = {3},
 Pages = {1498--1541},
 Year = {2018},
 Language = {English},
 DOI = {10.1214/17-AOP1208},
 zbMATH = {6894780},
 Zbl = {1451.60091}
}

@Book{HoermanderAnalysis,
 Author = {H{\"o}rmander, Lars},
 Title = {The analysis of linear partial differential operators. {I}. {Distribution} theory and {Fourier} analysis.},
 Edition = {2nd ed.},
 FSeries = {Grundlehren der Mathematischen Wissenschaften},
 Series = {Grundlehren Math. Wiss.},
 ISSN = {0072-7830},
 Volume = {256},
 ISBN = {3-540-52343-X; 3-540-52345-6},
 Year = {1990},
 Publisher = {Springer-Verlag},
 Address = {Berlin etc.},
 Language = {English},
 zbMATH = {47968},
 Zbl = {0712.35001}
}

@Article{RoecknerWangLogHarnack,
 Author = {R{\"o}ckner, Michael and Wang, Feng-Yu},
 Title = {Log-{Harnack} inequality for stochastic differential equations in {Hilbert} spaces and its consequences},
 FJournal = {Infinite Dimensional Analysis, Quantum Probability and Related Topics},
 Journal = {Infin. Dimens. Anal. Quantum Probab. Relat. Top.},
 ISSN = {0219-0257},
 Volume = {13},
 Number = {1},
 Pages = {27--37},
 Year = {2010},
 Language = {English},
 DOI = {10.1142/S0219025710003936},
 zbMATH = {5709649},
 Zbl = {1207.60053}
}

@incollection {BakrySobolev,
    AUTHOR = {Bakry, Dominique},
     TITLE = {On {S}obolev and logarithmic {S}obolev inequalities for
              {M}arkov semigroups},
 BOOKTITLE = {New trends in stochastic analysis ({C}haringworth, 1994)},
     PAGES = {43--75},
 PUBLISHER = {World Sci. Publ., River Edge, NJ},
      YEAR = {1997},
      ISBN = {981-02-2867-8},
   MRCLASS = {60J35},
  MRNUMBER = {1654503},
MRREVIEWER = {Kazuaki\ Taira},
}

@Article{WangExponentialContraction,
 Author = {Wang, Feng-Yu},
 Title = {Exponential contraction in {Wasserstein} distances for diffusion semigroups with negative curvature},
 FJournal = {Potential Analysis},
 Journal = {Potential Anal.},
 ISSN = {0926-2601},
 Volume = {53},
 Number = {3},
 Pages = {1123--1144},
 Year = {2020},
 Language = {English},
 DOI = {10.1007/s11118-019-09800-z},
 zbMATH = {7243163},
 Zbl = {1454.60129}
}

@Article{MicloHyperboundedness,
 Author = {Miclo, Laurent},
 Title = {On hyperboundedness and spectrum of {Markov} operators},
 FJournal = {Inventiones Mathematicae},
 Journal = {Invent. Math.},
 ISSN = {0020-9910},
 Volume = {200},
 Number = {1},
 Pages = {311--343},
 Year = {2015},
 Language = {English},
 DOI = {10.1007/s00222-014-0538-8},
 zbMATH = {6434689},
 Zbl = {1312.47016}
}

@Article{WangCriteriaSpectralGap,
 Author = {Wang, Feng-Yu},
 Title = {Criteria of spectral gap for {Markov} operators},
 FJournal = {Journal of Functional Analysis},
 Journal = {J. Funct. Anal.},
 ISSN = {0022-1236},
 Volume = {266},
 Number = {4},
 Pages = {2137--2152},
 Year = {2014},
 Language = {English},
 DOI = {10.1016/j.jfa.2013.11.016},
 zbMATH = {6299493},
 Zbl = {1298.47018}
}

\end{document}